\documentclass[amscd,amssymb,verbatim,10pt]{amsart}
\usepackage[fontsize = 10bp]{fontsize}
\usepackage{scalerel}
\usepackage{stackengine,wasysym}
\usepackage{bbm}

\usepackage{enumitem}
\setlist[enumerate]{itemsep=0pt}
\setlist[itemize]{itemsep=-0.5pt}

\newcommand\rwt[1]{\ThisStyle{%
  \setbox0=\hbox{$\SavedStyle#1$}%
  \stackengine{-.1\LMpt}{$\SavedStyle#1$}{%
    \stretchto{\scaleto{\SavedStyle\mkern.2mu\AC}{.465\wd0}}{.6\ht0}%
  }{O}{c}{F}{T}{S}%
}}

\usepackage{amssymb,latexsym, amsmath, amscd, array, graphicx, pb-diagram}
\usepackage{hyperref}
\usepackage{tikz-cd}

\usepackage{color}
\hypersetup{
    colorlinks=true,
    linkcolor=red,
    urlcolor=black,
    citecolor=blue
           }
           
\usepackage{amsthm}
\usepackage{hyperref}
\usepackage[all]{xy}
\usepackage[scr]{rsfso}

\swapnumbers \numberwithin{equation}{section}

\theoremstyle{plain}

\newtheorem{thm}{Theorem}[section]
\newtheorem{theorem}[thm]{Theorem}

\newtheorem{conjec}[thm]{Conjecture}
\newtheorem{prop}[thm]{Proposition}
\newtheorem{cor}[thm]{Corollary}

\theoremstyle{definition}
\newtheorem{defn}[thm]{Definition}
\newtheorem{definition}[thm]{Definition}

\newtheorem{rem}[thm]{Remark}
\newtheorem{remark}[thm]{Remark}

\newtheorem{ex}[thm]{Example}

\newtheorem{question}[thm]{Question}

 \newcommand{\Wi}{\widetilde}

\DeclareMathOperator{\cat}{{\rm cat}}
\DeclareMathOperator{\TC}{{\rm TC}}
\DeclareMathOperator{\dcat}{{\rm dcat}}
\DeclareMathOperator{\dTC}{{\rm dTC}}

\DeclareMathOperator{\cd}{{\rm cd}}

\DeclareMathOperator{\tr}{{\rm tr}}

\usepackage{enumitem}

\def\Int{\protect\operatorname{Int}}

\makeatletter
\def\@tocline#1#2#3#4#5#6#7{\relax
  \ifnum #1>\c@tocdepth 
  \else
    \par \addpenalty\@secpenalty\addvspace{#2}%
    \begingroup \hyphenpenalty\@M
    \@ifempty{#4}{%
      \@tempdima\csname r@tocindent\number#1\endcsname\relax
    }{%
      \@tempdima#4\relax
    }%
    \parindent\z@ \leftskip#3\relax \advance\leftskip\@tempdima\relax
    \rightskip\@pnumwidth plus4em \parfillskip-\@pnumwidth
    #5\leavevmode\hskip-\@tempdima
      \ifcase #1
       \or\or \hskip 1em \or \hskip 2em \else \hskip 3em \fi%
      #6\nobreak\relax
    \hfill\hbox to\@pnumwidth{\@tocpagenum{#7}}\par
    \nobreak
    \endgroup
  \fi}
\makeatother

\newcommand \pa[2]{\frac{\partial #1}{\partial #2}}

\def\C{{\mathbb C}}
\def\Z{{\mathbb Z}}
\def\Q{{\mathbb Q}}
\def\R{{\mathbb R}}
\def\N{{\mathbb N}}

\def\1{\hbox{\rm\rlap {1}\hskip.03in{\rom I}}}
\def\Bbbone{{\rm1\mathchoice{\kern-0.25em}{\kern-0.25em}
{\kern-0.2em}{\kern-0.2em}I}}

\def\pa{\partial}

\def\wt{\widetilde}

\def\ov{\overline}

\usepackage{comment}

\long\def\forget#1\forgotten{} %

\newcommand\ver[1]{\marginpar{\tiny Changed in Ver \VER}}

\date{\today}

\begin{document}

\begin{abstract}
We present a detailed study of the curvature and symplectic asphericity properties of symmetric products of surfaces. We show that these spaces can be used to answer nuanced questions arising in the study of closed Riemannian manifolds with positive scalar curvature. For example, we prove that symmetric products of surfaces \emph{sharply} distinguish between two distinct notions of \emph{macroscopic dimension} introduced by Gromov and the second-named author. As a natural generalization of this circle of ideas, we address the Gromov--Lawson and Gromov conjectures in the K\"ahler projective setting and draw new connections between the theories of the minimal model, positivity in algebraic geometry, and macroscopic dimensions.
\end{abstract}



\title[Curvature, macroscopic dimensions, and symmetric products]{Curvature, macroscopic dimensions, and symmetric products of surfaces}

\author[L.~F.~Di~Cerbo]{Luca~F.~Di~Cerbo}

\author[A.~Dranishnikov]{Alexander~Dranishnikov} 

\author[E.~Jauhari]{Ekansh~Jauhari}

\address{Luca F. Di Cerbo, Department of Mathematics, University
of Florida, 358 Little Hall, Gainesville, FL 32611-8105, USA.}

\email{ldicerbo@ufl.edu}

\address{Alexander N. Dranishnikov, Department of Mathematics, University
of Florida, 358 Little Hall, Gainesville, FL 32611-8105, USA.}

\email{dranish@ufl.edu}

\address{Ekansh Jauhari, Department of Mathematics, University
of Florida, 358 Little Hall, Gainesville, FL 32611-8105, USA.}

\email{ekanshjauhari@ufl.edu}



\maketitle
\tableofcontents

\section{Introduction}
The Fundamental Theorem of Algebra establishes an isomorphism between ordered $n$-tuples  and unordered $n$-tuples in the complex plane $\mathbb C$. Thus, the symmetric product $SP^n(\mathbb C)$ is diffeomorphic to the vector space $\mathbb C^n$. This brings to life complex $n$-manifolds $SP^n(M_g)$ for every closed orientable surface $M_g$ of genus $g$. These manifolds are extremely rich from both topological and algebraic geometric points of view. Indeed, during the last seven decades, such spaces have been extensively studied by topologists and algebraic geometers alike.

In algebraic topology, their study dates back at least to the important works of Steenrod and Dold~\cite{Do1}, Dold--Thom~\cite{DT}, and Macdonald~\cite{Mac}. This line of research was further developed by Milgram~\cite{Mi} and his school; see, for example,~\cite{Ka1},~\cite{KS}. Recently, certain symmetric products of surfaces played a role in the definition of Heegaard Floer homology, see~\cite{Szabo}. The algebraic topology of symmetric products of curves was also studied in connection with moduli spaces of gauged vortices on a closed Riemann surface, see~\cite{BR}. For an extended survey on symmetric products of surfaces in topology and physics, we refer to~\cite{Belgrade}.

In algebraic geometry, symmetric products of surfaces $SP^n(M_g)$ arise naturally as the smooth projective varieties parametrizing effective divisors of degree $n$ on the  Riemann surface $M_g$. Thus, their study is inevitably connected with the theory of algebraic curves and their Jacobians. We refer to the beautiful book of Arbarello \emph{et al.}~\cite{ACGH} for a panoramic view of the immeasurable algebraic geometry literature concerning symmetric products of surfaces.

Interestingly, these spaces are less well-studied from a Riemannian geometry point of view. This is somewhat surprising for the following reasons. Heuristically, for fixed $n$, $SP^n(M_g)$ becomes more negatively curved as $g$ increases. On the other hand, for fixed $g$, $SP^n(M_g)$ becomes more positively curved as $n$ increases. This fact makes them prime candidates for answering a variety of nuanced geometric questions as they exhibit subtler properties than product spaces.  Motivated by this observation, we study in detail the \emph{sectional}, \emph{holomorphic sectional}, \emph{Ricci}, and \emph{scalar} curvatures of such spaces. Importantly, we find that symmetric products of surfaces sharply distinguish two distinct notions of \emph{macroscopic dimension} introduced by Gromov in~\cite{Gr2} and the second-named author in~\cite{Dr2}. Recall that Gromov's notion of macroscopic dimension was introduced in order to study closed manifolds with positive scalar curvature, and it plays an important role in the \emph{Gromov--Lawson} and \emph{Gromov} conjectures. Naturally, we also discuss the \emph{essentiality} and \emph{inessentiality} of these spaces, the existence and non-existence of \emph{spin structures} on them and their universal covers, and derive several results concerning their \emph{Lusternik--Schnirelmann category} and \emph{topological complexity}. 

Motivated by the understanding of curvature and macroscopic dimensions of symmetric products of surfaces, we also address some general questions concerning this circle of ideas. First, we investigate the general problem of identifying large classes of Riemannian manifolds on which the two distinct notions of macroscopic dimension agree. This investigation is directly needed to quantify the sharpness of the examples produced by symmetric products of surfaces. Second, we address the Gromov--Lawson and Gromov conjectures for \emph{K\"ahler} metrics on smooth projective manifolds. In particular, we draw new connections between the theories of the minimal model, positivity in complex algebraic geometry, and macroscopic dimensions. The results provide support to these long-standing conjectures.

\subsection{Organization}

In Section~\ref{two}, we review some basic definitions and miscellaneous material. In particular, we discuss the Abel--Jacobi map from both an algebraic geometric and topological point of view. Section~\ref{chh} contains some standard and less standard results on the topology of symmetric products of curves. In particular, we provide an easy and direct computation of $\pi_2$ for symmetric products of surfaces in Propositions~\ref{pi2} and~\ref{pi2 not 0}. In Section~\ref{SA}, we show that many symmetric products of curves are symplectically aspherical non-aspherical manifolds with $\pi_2\neq 0$. Section~\ref{curvature} contains a detailed study of the curvature properties of symmetric products of curves, both in the Riemannian geometry and K\"ahler geometry settings. Among many other things, we complete the determination of the existence of K\"ahler metrics of non-positive holomorphic sectional curvature on $SP^n(M_g)$,  see Proposition~\ref{it can} and Remark~\ref{biswas sharp}. In Section~\ref{lstc}, Proposition~\ref{essential} identifies the symmetric product of surfaces that are \emph{rationally essential}. We also exactly compute the LS-category and topological complexity of all symmetric products of surfaces; see Theorems~\ref{lscat} and~\ref{tcthm}. Section~\ref{Cmacroscopic} studies the interactions between curvature and the two distinct notions of macroscopic dimension. In particular, we address the question of when these macroscopic dimensions agree for the universal covers of closed smooth manifolds; see Question~\ref{impques} and the answers in Propositions~\ref{Dranishnikov},~\ref{PRdim} and Theorems~\ref{top},~\ref{lowdim}. 

In Section~\ref{glsection}, we prove that a strengthening of Gromov's and Lawson's Conjecture~\ref{GLC} is true if we restrict our attention to projective manifolds that admit an aspherical smooth minimal model and are equipped with K\"ahler metrics. Along the way, we also compute the macroscopic dimension of such spaces. Finally, we prove a weak version of Gromov's Conjecture~\ref{G} for smooth projective varieties that admit K\"ahler metrics of positive scalar curvature. Section~\ref{spinstr} studies spin structures on $SP^n(M_g)$ and its universal cover; see Theorem~\ref{main2} for a complete classification. We then use Theorem~\ref{main2} to estimate the macroscopic dimensions of symmetric products of curves and to discuss the (non-)existence of Riemannian metrics of positive scalar curvature on them in Section~\ref{dimandpsc}. In the same section, we point out that the two distinct notions of macroscopic dimensions agree for aspherical manifolds but not for symplectically aspherical manifolds; see Theorem~\ref{todo4}, Corollary~\ref{nondim}, and Remark~\ref{mc and MC}. Similarly, these notions coincide on  Riemannian manifolds of non-positive sectional curvature but not for K\"ahler manifolds of non-positive holomorphic sectional curvature; see Proposition~\ref{it can} and Remarks~\ref{carthad},~\ref{mc and MC}.
Finally, we point out that our examples are dimensionally sharp. Indeed, the two macroscopic dimensions agree up to dimension $3$ but not in dimension $4$. This happens already on the symmetric squares $SP^2(M_g)$ for $g\geq 3$, see Theorem~\ref{lowdim} and Remark~\ref{mc and MC}. The result is somewhat surprising as projective surfaces are rather constrained from a topological point of view. 

We conclude this study with a partially speculative section. In Section~\ref{converse}, we discuss what the converse to Gromov's conjecture should be (if any!) and discuss the relevance of our examples. This section nevertheless contains positive results. Indeed, in Theorem~\ref{totallyconverse}, we prove that a converse to Gromov's conjecture holds true for certain totally non-spin manifolds having amenable fundamental groups. That said, there is a price to pay; the original notion of Gromov's macroscopic dimension needs to be replaced by the one considered by the second-named author.

\section{Preliminaries}\label{two}

In this section, we recall several topological and geometric notions and facts that we will be using in this paper.

\subsection{Symmetric products}
For a CW complex $X$ and integer $k\ge 1$, let $SP^k(X)$ denote the $k$-th symmetric power of $X$, obtained as the orbit space of the natural action of the symmetric group $S_k$ on the product $X^k$ that permutes the coordinates. There are basepoint inclusions $SP^k(X)\hookrightarrow SP^{k+1}(X)$. The colimit over the symmetric products $SP^k(X)$ defines the infinite symmetric product $SP^{\infty}(X)$. If $X$ is connected, then it follows from the Dold--Thom theorem~\cite{DT} that $\pi_n(SP^{\infty}(X))=\wt{H}_n(X)$ for each $n\ge 1$.

\begin{ex}
For each $n$, we have that $SP^n(\C P^1)=\C P^n$, $SP^n(S^1)$ is homotopic to $S^1$, and $SP^n(\R P^2)=\R P^{2n}$.
\end{ex}

Let $M_g$ denote the closed orientable surface of genus $g\ge 0$. It is well-known that $SP^n(M_g)$ is a closed orientable $2n$-manifold for each $n\ge 1$ (it follows, for example, from~\cite[Lemma 5]{KS}). In Section~\ref{chh}, we will look at various (co)homological and homotopical aspects of these spaces.

\subsection{Some basics from complex algebraic geometry}\label{alggeo} 
In Sections~\ref{SA},~\ref{curvature}, and~\ref{glsection}, we use some classical and more advanced notions and results from complex algebraic geometry. Ubiquitous throughout this study are \emph{nef}, \emph{pseudo-effective}, \emph{big}, and \emph{ample} line bundles and divisors over smooth projective varieties. We need both an algebraic and analytical description of the cones of such divisors. For the algebraic description of cones of line bundles over smooth algebraic varieties, we refer to the classical book of Lazarsfeld~\cite[Chapters I \& II]{Laz}. For the analytical description using currents, we refer to the beautiful book of Demailly~\cite[Chapters 6, 18 \& 19]{Demailly}. Also, the paper~\cite{DD15} provides a readable introduction on how the standard results in the minimal model program (such as Mori's cone theorem) can be used in K\"ahler geometry. Finally, we refer to both the books of Griffiths and Harris~\cite{GH} and Arbarello \emph{et al.}~\cite{ACGH} for the basic theory of algebraic curves and their Jacobians.

\subsection{Geometric Abel--Jacobi map}
In this section, we recall some basic definitions and results from the theory of the Abel--Jacobi map in complex algebraic geometry. Let $M_g$ be a complex curve of genus $g$. We denote by $J$ the Jacobian of $M_g$. The Abel--Jacobi map 
\[
\mu_1: M_{g}\to J
\]
is defined by setting for any $q\in M_g$
\[
\mu_1(q):=\left(\int^q_{p_0}\omega_1, ..., \int^q_{p_0}\omega_g\right)
\]
for a fixed base point $p_0\in M_g$, where $\omega_1, \ldots, \omega_g$ is a basis for $H^0(M_g, K_{M_{g}})$. Here, $\pi:K_{M_{g}}\to M_{g}$ denotes the canonical line bundle. The Abel--Jacobi map extends naturally to a holomorphic map for the symmetric product $SP^n(M_g)$ to $J$. Indeed, $SP^n(M_g)$ is the set of effective divisors of degree $n$ on $M_g$, i.e., a point in $SP^n(M_g)$ can be identified with a formal sum $\sum^n_{\lambda=1}p_{\lambda}$ of points $p_{\lambda}\in M_g$. The associated Abel--Jacobi map
\[
\mu_{n}: SP^n(M_{g})\to J
\]
is defined by setting for any $Q=\sum^n_{\lambda=1}p_{\lambda}\in SP^n(M_g)$
\[
\mu_n(Q):=\left(\sum_{\lambda}\int^{p_{\lambda}}_{p_0}\omega_1, \ldots, \sum_{\lambda}\int^{p_{\lambda}}_{p_0}\omega_g\right)
\]
for a fixed base point $p_0\in M_g$, where $\omega_1, \ldots, \omega_g$ is a basis for $H^0(M_g, K_{M_{g}})$. Similarly, the Abel--Jacobi map is defined for all divisors $D$ on $M_g$ of degree zero. More precisely, we denote by
\[
\text{Div}^{0}(M_g):=\big\{D\in \text{Div}(M_{g})\hspace{1mm} | \hspace{1mm} \text{deg}(D)=0\big\}
\]
the set of divisors with degree zero, and we define
\[
\mu: \text{Div}^0(M_g)\to J
\]
by setting
\[
\mu(D):=\left(\sum_{\lambda}\int^{q_{\lambda}}_{p_{\lambda}}\omega_1, \ldots, \sum_{\lambda}\int^{q_{\lambda}}_{p_{\lambda}}\omega_g\right),
\]
where $D=\sum_\lambda (q_{\lambda}-p_{\lambda})\in\text{Div}^{0}(M_g)$.

We conclude this section by recalling the statements of two classical theorems due to Abel and Jacobi, respectively, describing some important geometric features of the Abel--Jacobi maps. These results will be used in the remainder of this paper, and we state them here in the form that is most convenient for our purposes.

Abel's theorem characterizes divisors that are associated with meromorphic functions. Recall that because of Stokes' theorem, the zero and the pole sets of a meromorphic function give a divisor of degree zero.
\begin{theorem}[Abel]\label{Abel}
$D=\sum_\lambda (q_{\lambda}-p_{\lambda})\in\textup{Div}^{0}(M_g)$ is the divisor associated to a meromorphic function $f$ on $M_g$ if and only if $\mu(D)=0\in J$.
\end{theorem}

Jabobi's theorem gives that $\mu_n$ is a birational map for $n=g\geq 2$.

\begin{theorem}[Jacobi]\label{Jacobi}
The map $\mu_n: SP^n(M_n)\to J$ is surjective and generically one-to-one.
\end{theorem}

We refer to Chapter 2 in the classical book of Griffiths and Harris~\cite{GH} for more details.

\subsection{Topological Abel--Jacobi map} 
Let us now give a topological account of the Abel--Jacobi map. The surface $M_g$ can be viewed as a connected sum of $g$ copies of $2$-tori $T_i$ with a 2-sphere $S^2$. Collapsing $S^2$ with $g$ holes to a point defines a map 
\[
\Wi{q}:M_g\to\bigvee_{i=1}^g T_i
\]
onto the wedge of $g$ copies of 2-tori. For each torus $T_i$, we fix two circles $a_i$ and $b_i$, a parallel and a meridian missing the wedge point. We use the same notations for homology classes defined by $a_i$ and $b_i$. Moreover, we use the same notations for their images under the homomorphism $\xi_*:H_1(M_g)\to H_1(SP^n(M_g))$ induced by the basepoint inclusion $\xi:M_g\hookrightarrow SP^n(M_g)$.

The projections onto the summands define a map
$\psi:\bigvee T_i\to\prod T_i=J$. The Abel--Jacobi map $\mu_1:M_g\to J$ is the composition $\mu_1=\psi\circ \Wi{q}$. The additions in $SP^2(M_g)$ and in the $2g$-torus $J$
define the Abel--Jacobi map for $SP^2(M_g)$ making the following diagram commute:
\[
\begin{tikzcd}[contains/.style = {draw=none,"\in" description,sloped}]
M_g\times M_g \arrow{r} \arrow{d}{\mu_1\times\mu_1}
& 
SP^2(M_g) \arrow{d}{\mu_2}
\\
J \times J\arrow{r}
&
J.
\end{tikzcd}
\]
Similarly, we define $\mu_3:SP^3(M_g)\to J$ and so on. Finally, we get the Abel--Jacobi map $\mu_n:SP^n(M_g)\to J$ for each $n\ge 1$.

\section{Cohomology, homology, and homotopy}\label{chh}

In this section, we study in detail the cohomology, homology, and homotopy of symmetric products of curves. Among other things, we provide a simple computation of the second homotopy group of these spaces.

\subsection{Cohomology}\label{cohomo}
By $a_i^*$ and $b_i^*$, we denote the Hom dual to the cohomology classes $a_i$ and $b_i$ of $T_i$ for each $i\le g$. We use the same notation for their images under the projection map $J\to T_i$ and their images under $\mu_n^*$, the map induced in homology by the Abel--Jacobi map $\mu_n:SP^n(M_g)\to J=\prod T_i$. 

Let us consider the standard CW complex structure on $M_g=\bigvee_{2g}S^1\cup_\phi D^2$, where the circles in the wedge are indexed by the letters $a_i$ and $b_j$, and the attaching map $\phi$ is defined by the product of the commutators $[a_1,b_1]\cdots[a_g,b_g]$. Let 
\[
q:M_g\to M_g/\bigl(\bigvee_{2g}S^1\bigr)\to S^2
\]
and let 
\[
\bar q=SP^n(q):SP^n(M_g)\to SP^n(S^2)=\mathbb CP^n.
\]
Let $c$ denote the fundamental class of $M_g$ as well as its image in $H_2(SP^n(M_g))$ under the homomorphism induced by the base point inclusion $M_g\hookrightarrow SP^n(M_g)$. We recall from~\cite{Na},~\cite{Do1}, and~\cite{Do2} that for each $k\in\mathbb N$, the base point inclusion $SP^k(M_g)\to SP^{n+k}(M_g)$ induces a split monomorphism of homology and cohomology groups. Since $q$ takes the fundamental class to the fundamental class of $S^2$, the generator $c^*\in H^*(SP^n(S^2))$ goes under $(\bar{q})^*$ to the Hom dual to the cohomology class $c$, which will be denoted by $c^*$ as well. 

The integral cohomology of the torus $J=\prod T_i$ is given by the exterior algebra $H^*(J)=\Lambda(a_1^*,b_1^*,\dots,a_g^*,b_g^*)$ where $a_i^*,b_i^*$ generate $H^1(T_i)$ for each $1\le i \le g$. So, the map $(\mu_n,\bar q):SP^n(M_g)\to J\times\mathbb CP^n$ defines the ring homomorphism
\[
\Theta: \Lambda\left(a_1^*,b_1^*,\dots,a^*_g,b^*_g\right)\otimes\mathbb Z\left[c^*\right]\to H^*\left(SP^n(M_g)\right).
\]
I.~G.~Macdonald proved the following,~\cite{Mac}.

\begin{thm}\label{Mc}
The integral cohomology ring $H^*(SP^n(M_g))$ is the quotient of $H^*(J)\otimes\mathbb Z[c^*]$ by the following relation:
$$
a_{i_1}^*\cdots a_{i_l}^*b_{j_1}^*\cdots b_{j_m}^*(c^*-a_{k_1}^*b_{k_1}^*)\cdots(c^*-a_{k_r}^*b_{k_r}^*)(c^*)^s=0
$$
whenever $l+m+2r+s\ge n+1$ for any distinct set of indexes $i_1,\dots,i_l$, $j_1,\dots,j_m$, and $k_1,\dots, k_r$.
\end{thm}

We note that $(a_i^*)^2=(b_j^*)^2=0$, $a_i^*a_j^*=-a_j^*a_i^*$, and $b_i^*b_j^*=-b_j^*b_i^*$ for $i\ne j$. Also, $c^*$ commutes with $a_i^*$ and $b_j^*$.

\begin{rem}
When $n=2$, the relations are reduced to the following:
\begin{enumerate}
    \item $a_i^*b_j^*b_k^*=a_i^*a_j^*b_k^*=0$ for distinct indexes $i$, $j$, and $k$,
    \item $wc^*=0$ and for any $w$ equal to $a_i^*a_j^*$, $b_i^*b_j^*$, or $a_i^*b_j^*$ with $i\ne j$, and
    \item $w(c^*-a_k^*b_k^*)=0$ for any $k$ and for any $w$ equal to $a_i^*$, $b_j^*$, or $c^*$ with $i,j \ne k$.
\end{enumerate}

We note that $c^*$ is the Poincar\'e dual to the fundamental class $[M_g]$ and $(c^*)^2$ generates $H^4(SP^2(M_g))=\mathbb Z$. In fact, for each $n$, the $n$-th power $(c^*)^n$ generates the group $H^{2n}(SP^n(M_g))=\mathbb Z$.
\end{rem}

\begin{prop}\label{2g-torus}
In the integral cohomology ring $H^*(SP^n(M_g))$, we have that
\begin{enumerate}
\item the product $a_1^*b_1^*\cdots a_n^*b_n^*$ is non-zero for $n\le g$, and
\item the product $a_1^*b_1^*\cdots a_g^*b_g^*(c^*)^{n-g}$ is non-zero for $n\ge g$.
\end{enumerate}
\end{prop}

\begin{proof}
(1) By induction on $n$, we show that $a_1^*b_1^*\cdots a_n^*b_n^*=(c^*)^n\ne 0$ for $n\le g$.
This holds true for $n=1$. Assume that $a_1^*b_1^*\cdots a_{n-1}^*b_{n-1}^*=(c^*)^{n-1}$. Then
\[
a_1^*b_1^*\cdots a_n^*b_n^*=a_1^*b_1^*\cdots a_{n-1}^*b_{n-1}^*(a_n^*b_n^*-c^*)+ a_1^*b_1^*\cdots a_{n-1}^*b_{n-1}^*c^*
\]
\[
=(c^*)^{n-1}(a_n^*b_n^*-c^*)+(c^*)^n=(c^*)^n.
\]
Here we used Macdonald's relation $(c^*)^{n-1}(a_n^*b_n^*-c^*)=0$.

(2) Since the cohomology groups $H^*(SP^n(M_g))$ are torsion-free~\cite{Mac}, it suffices to prove (2) mod 2.
For $0\le k\le g$, we prove by induction on $k$ the following equality mod 2:
\[
(c^*)^{n-k}\prod_{s=1}^{k}a_{i_s}^*b_{i_s}^*=(c^*)^n.
\]
This is vacuously true for $k=0$.
Suppose that it holds true for $\ell<k$. Since $k\ge 1$, we have $2k+(n-k)\ge n+1$. Thus, we get by Macdonald's relations that
\[
(c^*)^{n-k}\prod_{s=1}^{k}(a_{i_s}^*b_{i_s}^*-c^*)=0.
\]
We now apply the induction hypothesis to obtain the mod 2 equality
\[
(c^*)^{n-k}\prod_{s=1}^{k}a_{i_s}^*b_{i_s}^*=\left(\sum_{\ell=1}^{k}(c^*)^{n-(k-1)}\prod_{s\ne \ell}a_{i_s}^*b_{i_s}^*\right)+\cdots+\sum_{s=1}^{k}(c^*)^{n-1}a_{i_s}^*b_{i_s}^*+(c^*)^n
\]
\[
= {k \choose k-1}(c^*)^{n}+{k\choose k-2}(c^*)^{n}+\cdots+{k\choose 1}(c^*)^{n}+(c^*)^{n}=(2^k-1)(c^*)^n=(c^*)^n.
\]
\end{proof}

The Chern classes of $SP^n(M_g)$ can be derived from the computations done in~\cite{Mat2} and~\cite{Mac}. More precisely, the first Chern class $c_1\in H^2(SP^n(M_g))$ is given in our notations as
\[
c_1=(n-g+1)c^*-\sum_{i=1}^ga_i^*b_i^*.
\]

\subsection{Homology}
The space $SP^\infty(X)$ is a free topological monoid. Thus, it enjoys Pontryagin's products for homology,~\cite{Ha}. By the Dold--Thom theorem~\cite{DT}, $SP^\infty(M_g)$ is homotopy equivalent to the Eilenberg--MacLane space $(S^1)^{2g}\times \mathbb CP^\infty$.

The Pontryagin ring of $SP^\infty(M_g)$ is the following graded algebra, first computed by H.~Cartan:
\[
H_*(SP^\infty(M_g))\cong\Lambda(a_1,b_1,\dots, a_g,b_g)\otimes\Gamma[c].
\]
Here, $\Gamma[c]$ denotes the divided polynomial algebra. Note that $\Gamma[c]$ is generated by the elements $c_k=\frac{c^k}{k!}$ 
and it is dual to the polynomial algebra $\mathbb Z[c^*]$,~\cite{Ha}.
J.~Milgram gave a bigraded description of that Pontryagin ring,~\cite{Mi}. The bigrading comes from the Steenrod splitting~\cite{Do2}
\[
H_*(SP^n(X))=\bigoplus_{m=1}^nH_*(SP^m(X),SP^{m-1}(X)),
\]
where $n\in\mathbb N\cup\{\infty\}$. Thus, $H_*(SP^n(M_g))$ is a bigraded subgroup of the ring $\Lambda(a_1,b_1,\dots, a_g,b_g)\otimes\Gamma[c]$.
We refer to~\cite{Ka1} for further details. 

The bigrading implies the following.


\begin{prop}\label{2-homology}
\begin{enumerate}
    \item The group $H_1(SP^n(M_g))$ is freely generated by $a_1,\dots, a_g$ and $b_1,\dots, b_g$, and the map $(\mu_n)_*:H_1(SP^n(M_g))\to H_1(J)$ induced by the Abel--Jacobi $\mu_n:SP^n(M_g)\to J$ map is an isomorphism.
    \item The integral homology group $H_2(SP^n(M_g))$ is freely generated by $c$ and the Pontryagin products $a_i\cdot a_j$ and $b_i\cdot b_j$ for $i<j$, and $a_i\cdot b_j$ for  all $i,j$.
\end{enumerate}
\end{prop}
\begin{cor}\label{mu}
The Abel--Jacobi map $\mu_n:SP^n(M_g)\to J$ induces an isomorphism of the fundamental groups.
\end{cor}
\begin{proof}
Since both fundamental groups are abelian, they are isomorphic to the respective first homology groups. So, the result follows from Proposition~\ref{2-homology} (1).
\end{proof}


\subsection{Second homotopy group}
For a CW complex $X$, the space $SP^n(X)$ does not have a natural CW complex structure. When $X$ is a 2-dimensional complex with only one vertex, Kallel and Salvatore defined a natural homotopy equivalence $SP^n(X)\to\overline{SP}^n(X)$ to a CW complex $\overline{SP}^n(X)$,~\cite{KS}. 
All important features of $SP^n(X)$ transfer to $\overline{SP}^n(X)$. In particular, there are the canonical base point inclusions $\overline{SP}^n(X)\to \overline{SP}^{n+1}(X)$.

The main feature of this construction is that the quotient map $X^n\to\overline{SP}^n(X)$ is cellular for the product CW structure on $X^n$. This implies
the following remarkable feature of the CW structure on $\overline{SP}^n(X)$:

$(\ast \ast)$ For each $k\le n$, the $k$-skeleton of $\overline{SP}^n(X)$  coincides with the $k$-skeleton of $\overline{SP}^k(X)$.

\begin{prop}\label{pi2}
For $n\ge 3$, $\pi_2(SP^n(M_g))=\mathbb Z$.
\end{prop}

\begin{proof}
The universal covering $p:X\to\overline{SP}^\infty(M_g)$ has a CW structure induced from $\overline{SP}^\infty(M_g)$.
Since it is homotopy equivalent to the universal covering of $\mathbb CP^\infty\times T^{2g}$, it is homotopy equivalent to $\mathbb CP^\infty$. 
The universal covering $Y$ of $\overline{SP}^n(M_g)$ is the pullback of $p$ with respect to the base point inclusion
$\overline{SP}^n(M_g)\subset\overline{SP}^\infty(M_g)$. Hence, in view of $(\ast \ast)$, $Y$ satisfies $Y^{(3)}=X^{(3)}$. Thus,
we have a chain of equalities
\[
\pi_2(SP^n(M_g))=\pi_2(Y)=\pi_2(Y^{(3)})=\pi_2(X^{(3)})=\pi_2(X)=\pi_2(\mathbb CP^{\infty})=\mathbb Z.
\]
\end{proof}

Recall that a rational curve in $M$ is simply the image of a non-constant holomorphic map $f: \C P^1\to M$.

The following result generalizes~\cite[Lemma 9]{Ka2}.
\begin{prop}\label{pi2 not 0}
For each $g\geq 0$, the second homotopy group $\pi_2(SP^2(M_g))$ is a $\mathbb Z^{2g}$-module generated by a single element.
\end{prop}

\begin{proof}
When $g=0, 1$, the result is standard and the $\mathbb Z^{2g}$-module is trivial. For $g\geq 2$, it is classically known (see, for instance,~\cite[Example 3.3]{DCP}), that if we select a hyperelliptic structure on $M_g$, then there is a unique smooth rational curve in $SP^{2}(M_g)$ contracted by the Abel--Jacobi map. The Abel--Jacobi map is a biholomorphism outside this rational curve. By composing with a translation of the Albanese torus, we may always assume such a rational curve is contracted to the origin in $J$. Since the Abel--Jacobi map gives an isomorphism on $\pi_1$ (\emph{cf.} Corollary~\ref{mu}), by lifting this map to the universal covers, we conclude the proof. Indeed, we construct infinitely many disjoint smooth rational curves in $\rwt{SP^2(M_g)}$ that are the preimages of the lattice points in $\C^g$ corresponding to $\pi_1(J)=\Z^{2g}$.
\end{proof}

\begin{remark}
The second homotopy group of symmetric products of curves can be derived from the computation of the homology groups of their universal covers done by Bökstedt and Romão in~\cite{BR}. Our computation here is more direct.
\end{remark}

\section{Symmetric products of curves and symplectic asphericity}\label{SA}

In this section, we show that many symmetric products of curves are symplectically aspherical. First, let us recall the notion of symplectic asphericity.

Let $(M,\omega)$ be a closed symplectic manifold of dimension $2n$ (for $n\ge 1$), and let 
$[\omega]\in H^2(M;\R)$ be the de~Rham cohomology class corresponding to 
$\omega$. Let $h:\pi_2(M)\to H_2(M;\Z)$ denote the Hurewicz homomorphism. We call $(M,\omega)$ \emph{symplectically aspherical} (or \emph{SA} for short) if the composition $[\omega]\circ h$ is trivial. Equivalently, $(M,\omega)$ is SA if for any smooth map $f:S^2\to M$, we have that
\[
\int_{S^2} f^*\omega = 0.
\]
In this case, we call the symplectic form $\omega$ \emph{aspherical}.

We note that due to the Hurewicz theorem, a simply connected symplectic manifold cannot be SA. Some obvious examples of SA manifolds include symplectic manifolds that are aspherical. The \emph{first} non-aspherical SA examples were constructed by Gompf in~\cite{Go2} by using some branched covering constructions. 

It is well-known that SA manifolds are rationally essential, see~\cite{RO} (and also~\cite[Lemma 2.1]{RT}).


\begin{theorem}\label{todo4}
For $g\ge 2n-1$, the smooth manifold $SP^n(M_g)$ is symplectically aspherical.
\end{theorem}

\begin{proof}
Consider $M_g$ as a projective curve with a \emph{generic} complex structure in the moduli space $\mathcal{M}_g$. It is classically known (see, for example,~\cite[Page 261]{GH}) that the generic Riemann surface of genus $g$ is expressible as a branched cover of $\C P^1$ with
\begin{equation}\label{Riemann}
k=\biggl\lfloor{\frac{g+1}{2}\biggr\rfloor}+1
\end{equation}
sheets but no fewer. This implies that the map
\[
\mu_{n}: SP^n(M_g)\to J=\C^{2g}/\Lambda=T^{2g}, \quad g\geq 2n-1,
\]
is injective and therefore an embedding. To see this, assume that 
\[
\mu_{n}(Q)=\mu_{n}(Q^\prime)
\]
for $Q=\sum^n_{\lambda=1}p_{\lambda}$, $Q^{\prime}=\sum^n_{\lambda=1}p^\prime_{\lambda}\in SP^n(M_g)$. We then have that $Q-Q^{\prime}\in \text{Div}^{0}(M_g)$ with
\[
\mu(Q-Q^\prime)=0\in J.
\]
Thus, we conclude by Theorem~\ref{Abel} that $Q-Q^{\prime}=(f)$, where $(f)$ is the divisor associated with a mermorphic function $f$ on $M_g$. By sending the poles of $f$ to the north pole of the Riemann sphere $\C P^1$, $f$ extends to a holomorphic branched covering map from $M_g$ to $\C P^1$ with $n$ sheets. Since $M_g$ is generic, this contradicts the bound in~\eqref{Riemann}. This shows that for $g\geq 2n-1$ and $M_g$ generic in $\mathcal{M}_g$, the Abel--Jacobi map $\mu:SP^n(M_g)\to \C^{2g}/\Lambda=T^{2g}$ is an embedding. We can now use this fact to produce symplectically aspherical K\"ahler classes on $SP^n(M_g)$. Let $\omega$ be a K\"ahler form on the Abelian variety $J=T^{2g}$. We take the pullback of $\omega$ along $\mu_n$ to get a K\"ahler metric on $SP^n(M_g)$ corresponding to $\mu_n^*\omega$. Since $\omega$ is an aspherical form, so is $\mu_n^*\omega$ by~\cite[Proposition 2.1]{KRT2}. Hence, for any $n$ and $g$ satisfying $g\geq2n-1$ and $M_g$ generic, the manifold $(SP^n(M_g),\mu_n^*\omega)$ is SA for any K\"ahler form $\omega$ on $J$.
\end{proof}

We note that the lower bound in the above proposition is sharp. Indeed, $SP^2(M_2)$, which is diffeomorphic to the blow-up $T^4\# \ov{\C P^2}$ of a $4$-torus at one point, is not symplectically aspherical due to~\cite[Proposition 4.10]{Ku}.

\begin{remark}
From Propositions~\ref{pi2} and~\ref{pi2 not 0}, we have that the group $\pi_2(SP^n(M_g))$ is infinite for each $n$ and $g$. Thus, the SA examples given in Theorem~\ref{todo4} are not aspherical. Recall that the \emph{first} non-aspherical SA examples were constructed by Gompf  in~\cite{Go2} by using branched covering constructions. We believe the examples given in Theorem~\ref{todo4} to be more natural and, in many ways, more classical.
\end{remark}

\section{Curvatures and symmetric products of curves}\label{Curvatures and symmetric products of curves}\label{curvature}

In this section, we study the curvature properties of symmetric products of curves, both in the Riemannian geometry and K\"ahler geometry settings. We begin by observing that symmetric products of curves cannot be non-positively curved.

\begin{prop}\label{nonpositive}
For any $n$ and $g$, the smooth manifold $SP^n(M_g)$ cannot support a Riemannian metric of non-positive sectional curvature.
\end{prop}
\begin{proof}
In Propositions~\ref{pi2} and~\ref{pi2 not 0}, we computed $\pi_2(SP^n(M_g))$ for all $n$ and $g$. This group turns out to be always infinite. By the Cartan--Hadamard theorem (see, for example,~\cite[Theorem 6.2.2]{Petersen}) on the other hand, a closed non-positively curved Riemannian manifold has vanishing $\pi_k$ for each $k\geq 2$ since the universal cover is diffeomorphic to $\R^n$.
\end{proof}

\subsection{The range $g\ge n$}
While $SP^n(M_g)$ cannot be non-positively curved in the Riemannian sectional curvature sense by Proposition~\ref{nonpositive}, our next result shows that for certain ranges of $n$ and $g$ and for certain complex structures on $M_g$, $SP^n(M_g)$ admits K\"ahler metrics with non-positive \emph{holomorphic} sectional curvature.

\begin{prop}\label{it can}
For $g\ge 2n-1$ and $M_g$ generic in the moduli space $\mathcal{M}_g$, the projective manifold $SP^n(M_g)$ supports K\"ahler metrics of non-positive holomorphic sectional curvature.
\end{prop}
\begin{proof}
Under these assumptions, the proof of Theorem~\ref{todo4} tells us that
\[
\mu_n: SP^n(M_g)\to J
\]
is an embedding. Let $\omega_0$ be the standard Euclidean flat metric on $J$. Since the flat metric has vanishing Riemannian curvature tensor, we can prove that the K\"ahler metric $\mu^*_n\omega_0$ on $SP^n(M_g)$ has non-positive holomorphic sectional curvature. Indeed, let $B$ be the second fundamental form of $SP^n(M_g)$ in $J$, let $R$ be the Riemannian curvature tensor of  $\mu^*_n\omega_0$, and let $I$ be the complex structure on $SP^n(M_{g})$. A direct computation yields that for any tangent vector $v$ on $SP^n(M_g)$, one has
\begin{equation}\label{secondF}
R(v, Iv, v, Iv)=-|B(v, v)|-|B(v, Iv)|\leq 0.
\end{equation}
This shows that the holomorphic sectional curvature of $\mu^*_n\omega_0$ is non-positive. For more details, see~\cite[Equation 9]{Goldberg}.
\end{proof}

\begin{remark}
The proof of Proposition \ref{it can} generalizes to show that $SP^n(M_g)$ supports K\"ahler metrics of non-positive holomorphic \textit{bisectional} curvature. Recall that the holomorphic bisectional curvature is defined as $R(v, Iv, w, Iw)$ for tangent vectors $v, w$. In particular, it reduces to the holomorphic sectional curvature when $v=w$. For more details, see again \cite[Equation 9]{Goldberg}.
\end{remark}

\begin{remark}\label{biswas sharp}
Our result in Proposition~\ref{it can} is sharp because if $M_g$ is generic, then $SP^n(M_g)$ cannot support a K\"ahler metric of non-positive holomorphic sectional curvature for any $n\ge \lfloor \tfrac{g+1}{2}\rfloor+1$ due to~\cite[Proposition 1.2]{Biswas}.
\end{remark}

\begin{prop}
For $g\ge 2n-1$ and $M_g$ generic in the moduli space $\mathcal{M}_g$, the projective manifold $SP^n(M_g)$ has Stein universal cover.
\end{prop}

\begin{proof}
Under these assumptions, the proof of Theorem~\ref{todo4} tells us that
\begin{equation}\label{embedding}
\mu_n: SP^n(M_g)\to J
\end{equation}
is an embedding. By Corollary~\ref{mu}, the induced map 
\[
(\mu_n)_*:\pi_1(SP^n(M_g))\to \pi_1(J)=\Z^{2g}
\]
is an isomorphism. Thus, the pullback of the universal covering map $\pi:\C^{2g}\to J$ is the universal covering map $\rwt{\pi}: \rwt{SP^n(M_g)}\to SP^n(M_g)$, and the map in~\eqref{embedding} lifts to a holomorphic embedding
\[
\Wi{\mu_n}: \rwt{SP^n(M_g)}\to \C^{2g}.
\]
By using the restriction of the holomorphic coordinate functions on $\C^{2g}$ to its subspace $\rwt{SP^n(M_g)}$, we can easily verify that  $\rwt{SP^n(M_g)}$ is Stein (see, for example,~\cite[Definition 2.2.1]{Forstneric}).
\end{proof}

\begin{theorem}\label{Mori-klt}\label{ample}
For $g\ge 2n-1$ and $M_g$ generic in $\mathcal{M}_g$, the projective manifold $SP^n(M_g)$ has ample canonical line bundle $K_{SP^n(M_g)}$.
\end{theorem}

\begin{proof}
By Theorem~\ref{todo4}, $SP^n(M_g)$ embeds in an abelian variety, so it cannot support rational curves. Since $SP^{n}(M_g)$ does not contain rational curves, Mori's Cone Theorem~\cite[Theorem 1.24]{Mori} tells us that $K_{SP^n(M_g)}$ is \emph{nef}. Moreover, it is known (see, for example,~\cite[Example, Page 39]{Abr}) that for $g>n$, the complex manifold $SP^n(M_g)$ is of general type, i.e., $K_{SP^n(M_g)}$ is \emph{big}. Since $K_{SP^n(M_g)}$ is big and nef, Shokurov--Kawamata's basepoint-free theorem~\cite[Theorem 3.3]{Mori} implies that $K_{SP^n(M_g)}$ is semi-ample, i.e., for $n$ large enough the map $\varphi_{|nK_{SP^n(M_g)}|}$ associated to the linear system $|nK_{SP^n(M_g)}|$ is defined everywhere. Thus, if $K_{SP^n(M_g)}$ is not ample, there exists a subvariety contracted by $\varphi_{|nK_{SP^n(M_g)}|}$. We then have a curve $C$ in $SP^n(M_g)$ such that 
\[
K_{SP^n(M_g)}\cdot C=0.
\]
Now, given an ample divisor $A$, since the cone of big divisors is open, for a rational number $\varepsilon>0$ small enough, we have that $K_{SP^n(M_g)}-\varepsilon A$ is a big $\Q$-divisor. Thus, there exists $m\in \N$ large enough so that $m(K_{SP^n(M_g)}-\varepsilon A)$ is linearly equivalent to an effective divisor say $E$. So, for all rational numbers $\delta>0$ small enough, we have that $(SP^n(M_g), \Delta)$ with $\Delta:=\delta E$ is a \emph{klt} pair, see~\cite[Corollary 2.35]{Mori}. Moreover, we have
\[
(K_{SP^n(M_g)}+\Delta)\cdot C = \Delta\cdot C = \delta m (K_{SP^n(M_g)}-\varepsilon A)\cdot C = -\varepsilon m\delta A\cdot C<0.
\]
By the Cone Theorem for klt pairs~\cite[Theorem 3.7]{Mori}, since $K_{SP^n(M_g)}+\Delta$ is not nef, we have at least one rational curve in $SP^n(M_g)$. We arrive at a contradiction and the proof is complete.
\end{proof}

\begin{remark}
On page 267 of~\cite{Gr91}, M.~Gromov asks if a K\"ahler hyperbolic manifold necessarily has an ample canonical line bundle. The answer is yes as it follows from Mori's theory along the lines of the proof of Theorem~\ref{ample}. Indeed, in~\cite{Gr91} it is shown that such manifolds are projective of general type and, of course, with no rational curves. This result is folklore but since it seems not to be widely known, we add this remark.
\end{remark}

\begin{prop}\label{Einstein}
For $g\ge 2n-1$ and $M_g$ generic in $\mathcal{M}_g$, the projective manifold $SP^n(M_g)$ admits a K\"ahler--Einstein metric with negative Ricci curvature.
\end{prop}

\begin{proof}
By Theorem~\ref{todo4}, we can find a K\"ahler form representing the integer cohomology class $-c_1(SP^n(M_g))=c_1(K_{SP^n(M_g)})$. By Aubin--Yau's celebrated theorem (see, for example,~\cite{Yau}), there exists a K\"ahler metric $\omega$ whose Ricci tensor is proportional to the metric, i.e., 
\[
Ric_{\omega}=\lambda\omega,\quad  \lambda<0.
\]   
So, this metric has constant and negative Ricci curvature. Thus, it is an Einstein metric with negative Ricci curvature.
\end{proof}

\begin{remark}\label{fulllist}
Proposition~\ref{Einstein} and Hitchin's obstruction to the existence of Einstein metrics in real dimension $4$ provide a complete list of the symmetric squares of curves that admit Einstein metrics. Concretely, $SP^2(M_0)$ supports the Fubini--Study Einstein metric, $SP^2(M_g)$ supports an Aubin--Yau type Einstein metric for $g\geq 3$ (\emph{cf.} Proposition~\ref{Einstein}), and $SP^2(M_g)$ cannot support any Einstein metric for $g\in\{1,2\}$ as it follows from~\cite[Theorem 1]{Hit74}.
\end{remark}

Another manifestation of the non-positive curvature properties of the symmetric product of curves with $g\geq n$ is given by the following.

\begin{prop}
For $g\geq n$, the smooth manifold $SP^n(M_g)$ cannot support a Riemannian metric of non-negative Ricci curvature.
\end{prop}
\begin{proof}
By Bochner's theorem (see, for example,~\cite[Corollary 7.3.15]{Petersen}), if a closed $m$-manifold $M$ supports a metric of non-negative Ricci curvature, then we have the estimate on the first Betti number
\[
b_1(M)\leq m,
\]
with equality holding if and only if $M$ is a flat $m$-torus. Since $\pi_1(SP^n(M_g))=\Z^{2g}$, we conclude that $SP^n(M_g)$ cannot support a metric of non-negative Ricci curvature for $g>n$. It remains to study the case $g=n$. Recall that the group $\pi_2(SP^n(M_g))$ is infinite for each $n$ and $g$ by Propositions~\ref{pi2} and~\ref{pi2 not 0}. This implies, in particular, that $SP^n(M_n)$ cannot be the $2n$-torus and the result follows.
\end{proof}

If we restrict our attention to K\"ahler metrics \emph{only}, the following proposition asserts that for $g\geq n$, the average scalar curvature on $SP^n(M_g)$ has to be negative. In particular, this implies that the scalar curvature of such metrics can never be positive. In Section~\ref{macdimsec}, we will address the \emph{harder} question of the existence of general Riemannian metrics of positive scalar curvature on these spaces. 

\begin{prop}\label{avgscalar}
For $g\geq n$ and for any K\"ahler metric $g_{\omega}$ on the projective manifold $SP^n(M_g)$, we have
\[
\int_{SP^n(M_g)}s_{g_{\omega}}d\mu_{g_{\omega}}<0,
\]
where $s_{g_{\omega}}: SP^n(M_g)\to\R$ is the scalar curvature.
\end{prop}

\begin{proof}
A beautiful result of Yau~\cite[Corollary 2]{Yau74} tells us that on a K\"ahler manifold $(M, g_{\omega})$, the inequality
\begin{equation}\label{integral inequality}
    \int_{M}s_{g_{\omega}}d\mu_{g_{\omega}}\geq 0
\end{equation}
implies that either all the plurigenera $P_{m}:=\dim H^0(M, K^m_{M})$ vanish, or the first Chern class $c_{1}(M)\in H^2(M; \Z)$ is a torsion class. It is known (see, for example,~\cite[Example, Page 39]{Abr}) that for $g>n$, the complex manifold $SP^n(M_g)$ is of general type. This implies that $P_{m}=\dim H^0(SP^n(M_g), K^m_{SP^n(M_g)})$ grows asymptotically like $m^{n}$ so that most plurigenera of $SP^n(M_g)$ are non-vanishing. When $n=g$, since $K_{J}=\mathcal{O}_J$, we have that 
\[
K_{SP^n(M_n)}=E,
\]
where $E$ is an effective divisor coming from the exceptional locus of the Abel--Jacobi map $\mu_n: SP^n(M_n)\to J$. Recall that by Theorem~\ref{Jacobi} (Jacobi's theorem), $\mu_n$ is a birational morphism that cannot be the identity because $\pi_2(SP^n(M_n))$ is an infinite group. Since $E$ is effective, we have that $\dim H^{0}(SP^n(M_n), K_{SP^n(M_n)})\neq 0$. Now,
\[
c_{1}(SP^n(M_n))=-c_1(K_{SP^n(M_n)})=-PD(E)
\] 
where $PD(E)$ is the Poincar\'e dual of the effective divisor $E$. Finally, since any K\"ahler metric integrates non-trivially over $E$, we have that $c_{1}(SP^n(M_n))$ is not a torsion class. Hence,~\eqref{integral inequality} is not true in the case $M=SP^n(M_g)$ for any $g\ge n$.
\end{proof}

\subsection{The range $n\ge g$}
In the range $n>g$, the symmetric products of curves tend to be more positively curved. In the extreme case $g=0$, we have $SP^n(M_0)=\C P^n$ that admits a locally symmetric K\"ahler--Einstein metric of positive sectional curvature: the Fubini--Study metric. In a recent paper, Biswas extended this observation by showing that for $g\leq 1$, $SP^n(M_g)$ admits a K\"ahler metric with non-negative holomorphic \emph{bisectional} curvature, see~\cite[Theorem 1.1]{Biswas}. 

We now explore in depth how positively curved these symmetric products of curves can be. We begin by showing that in the larger range $n\ge 2g-1$, we can construct K\"ahler metrics with positive scalar curvature. When $n\ge 2g-1$, there is a complex vector bundle $E\to J$ (of rank $k\geq 2$)  such that the projectivization $\mathbb P E\to J$ with fiber  $\C P^{k-1}$ is isomorphic to the Abel--Jacobi map 
\[
\mu_n: SP^n(M_g)\to J.
\]
For this classical result, we refer to the textbook of Arbarello \emph{et al.}~\cite[Proposition 2.1 (Page 309)]{ACGH}. This fact was proved first in the paper by Arthur Mattuck~\cite{Mat1}
who had first computed the Chern classes of $E$,~\cite{Mat2}. Recall also that Macdonald \cite{Mac} computed the total Chern class of $SP^{n}(M_g)$ as
\[
c(SP^{n}(M_g))=(1+c^*)^{n-2g+1}\prod^{g}_{i=1}(1+c^*-a^{*}_{i}b^{*}_{i})
\]
where $c^*$, $a^*_i$, and $b^*_{i}$ are as in Section~\ref{chh}. For brevity, let us set $\theta:=\sum_{i=1}^ga_i^*b_i^*$. 

We now use the Chern classes to prove the following non-splitting phenomenon.

\begin{prop}\label{no product}
Suppose $n\ge 2g-1$ and $g>1$. Then neither $SP^n(M_g)$ nor any of its finite covers are homeomorphic to the product  $\mathbb C P^{n-g}\times T^{2g}$.
\end{prop}

\begin{proof}
We proceed by contradiction. Suppose $f:M'\to SP^n(M_g)$ is a finite covering map with $M'=\C P^{n-g}\times T^{2g}$. Since the fundamental group of $SP^n(M_g)$ is abelian, $f$ is a regular covering with action of a finite abelian group $G$. 
The main property of the transfer homomorphism $\tr:H^*(M')\to H^*(SP^n(M_g))$ is that the composition 
\[
\tr\circ f^*:H^*(SP^n(M_g))\to H^*(SP^n(M_g))
\]
is the multiplication by $|G|$,~\cite[Sections III.9 \& III.10]{Br}. Since $H^*(SP^n(M_g))$ is torsion free~\cite{Mac}, the induced homomorphism $f^*:H^*(SP^n(M_g))\to H^*(M')$ is injective. If $x\in H^2(\C P^{n-g})$ denotes the generator, then the cohomology ring of $\C P^{n-g}\times T^{2g}$ is the tensor product 
\[
\mathbb Z[x]/(x^{n-g+1})\otimes\Lambda(a_1^*,b_1^*,\dots,a_g^*,b_g^*).
\]
%
A covering map takes Chern classes to Chern classes. Hence, for each $k\ge 1$, $c_k(M') = f^*(c_k(SP^n(M_g)))$. The first Chern class of $SP^n(M_g)$ is given by
\[
c_1(SP^n(M_g))=(n-g+1)c^*-\theta,
\]
and the first Chern class of the cover $M'$ is
\[
c_1(M')=c_1(\C P^{n-g}\times T^{2g})=(n-g+1)\bar x,
\]
where $\bar x=x\otimes 1 \in H^2(\C P^{n-g}\times T^{2g})$. Thus, $(n-g+1)f^*c^*-f^*\theta=(n-g+1)\bar x$ and, hence
$$
\bar x=f^*(c^*)-\frac{1}{n-g+1}f^*(\theta).
$$
We consider the case $g>1$, which implies $n\ge 3$.
Since $(a_i^*b_i^*)^2=0$ for each $i$ and $(a_i^*b_i^*)(a_j^*b_j^*)=(a_j^*b_j^*)(a_i^*b_i^*)$ for each $i\ne j$, we obtain for $g>1$ that
\[
\prod_{i=1}^ga_i^*b_i^*=\frac{\theta^2}{2}.
\]
The second Chern class of $SP^n(M_g)$ is given by the formula
\begin{equation}\label{3'}
c_2(SP^n(M_g))=\frac{(n-g+1)(n-g)}{2}(c^*)^2 - (n-g)c^*\theta + \frac{\theta^2}{2},
\end{equation}
and the second Chern class of $M'$ is
\begin{equation}\label{2}
c_2(\C P^{n-g}\times T^{2g})=\frac{(n-g+1)(n-g)}{2}\bar x^2.
\end{equation}
The equation $f^*(c_2(SP^n(M_g)))=c_2(M')$ turns into the following:
\[
\frac{(n-g+1)(n-g)}{2}f^*((c^*)^2) - (n-g)f^*(c^*\theta) + \frac{1}{2}f^*(\theta^2)
\]
\[
=\frac{(n-g+1)(n-g)}{2}\left(f^*(c^*)-\frac{1}{n-g+1}f^*(\theta)\right)^2.
\]
For $n>2$, since the elements $(c^*)^2$, $c^*\theta$, $\theta^2$ are linearly independent and $f^*$ is a rational monomorphism, $f^*((c^*)^2)$, $f^*(c^*\theta)$, $f^*(\theta^2)$ are also linearly independent.
Therefore, we have the equalities of coefficients. For the coefficients of $f^*((c^*)^2)$ and $f^*(c^*\theta)$, they are honest equalities for all $n$ and $g$. For the coefficients of $f^*(\theta^2)$, we obtain the equation
\[
\frac{(n-g+1)(n-g)}{2(n-g+1)^2}=\frac{1}{2} \implies \frac{n-g}{n-g+1}=1,
\]
which has no solution. This completes the proof.
\end{proof}

The condition $g>1$ in Proposition~\ref{no product} is important in view of the following.
\begin{prop}\label{exception}
For any $n$, there is an $n$-fold covering $\mathbb CP^{n-1}\times T^2\to SP^n(T^2)$\footnote{\hspace{0.5mm}A different proof of this fact for $n=2$ was suggested by Will Sawin on the MathOverflow website, see \url{https://mathoverflow.net/q/476086}.}.
\end{prop}
\begin{proof}
Let $\xi:E\to T^2$ be an $n$-dimensional complex vector bundle whose projectivization is the Abel--Jacobi map $\mu:SP^n(T^2)\to T^2$.
By dimensional reasons, $\xi$ has $(n-1)$-linearly independent sections. 
Therefore, $\xi=L\oplus(n-1)\varepsilon$, where $\varepsilon$ is the trivial linear bundle over $T^2$ and $L$ is a linear bundle with the first Chern class $c_1(L)$. For any $n$-fold covering $f:T^2\to T^2$, we have for the first Chern class of the pullback that 
\[
c_1(f^*E)=c_1(f^*L)=nc_1(L)=c_1(nL).
\]
Due to dimensional reasons, complex vector bundles over $T^2$ are classified by the first Chern class. 
Therefore, $f^*E$ is isomorphic to $nL$.
Clearly, the projectivization of the Whitney sum of $n$ copies of the same line bundle $L$ is a trivial bundle with the fiber $\mathbb CP^{n-1}$. 
Then the pullback of $f$ with respect to $\mu$ is our required covering map. 
\end{proof}

Next, we study the scalar and Ricci curvatures in this range. 

\begin{thm}\label{n ge 2g}
For $n\ge 2g-1$, the manifold $SP^n(M_g)$ admits K\"ahler metrics of positive scalar curvature.
\end{thm}

\begin{proof}
Recall that in this range, the Abel--Jacobi map $\mu_n$ is isomorphic to the projectivization $\mathbb P E\to J$. Thus, the manifold $SP^n(M_g)$ can then be equipped with a one-parameter family of K\"ahler metrics with positive scalar curvature in the following way. Consider the one-parameter family of forms $\omega_t$ on $\mathbb P E$ defined by
\[
\omega_{t}:=\mu^*_n\omega+t\omega_{FS},
\]
where $\omega$ is any K\"ahler form on $J$, $t$ is a small real parameter, and $\omega_{FS}$ is the Fubiny--Study K\"ahler metric on the fiber $\C P^{k-1}$. Following Yau's computations in~\cite[Proposition 1]{Yau74}, we have that for $t$ small enough, $\omega_{t}$ is a K\"ahler metric on  $\mathbb P E$ with positive scalar curvature.
\end{proof}

We observe that in the range $0<g<n$, we still have some obstructions to positive curvature. 

\begin{prop}
For $g+1\leq n \leq 2g-2$, the smooth manifold $SP^n(M_g)$ cannot support a Riemannian metric of non-negative Ricci curvature.
\end{prop}
\begin{proof}
If we assume the existence of a Riemannian metric with non-negative Ricci curvature, then it follows from the splitting theorem of Cheeger and Gromoll~\cite[Theorem 3]{CG71} that the universal Riemannian cover $\rwt{SP^n(M_{g})}$ splits \emph{isometrically} as
\begin{equation}\label{cheeger}
\rwt{SP^n(M_g)}= N^{2(n-g)}\times \R^{2g},
\end{equation}
where $N^{2(n-g)}$ is a simply connected compact manifold with non-negative Ricci curvature. Recall that the Abel--Jacobi map
\[
\mu_n: SP^n(M_g) \to J=\C^g/\Lambda
\]
induces an isomorphism on $\pi_1$ and its generic fiber is $\C P^{n-g}$. Thus, the universal cover of $SP^n(M_g)$ is simply the pullback of the universal cover of the Jacobian $J$ via the map $\mu_n$. This implies that the fibers of the universal covering map
\[
\Wi{\mu_n}: \rwt{SP^n(M_g)}\to \C^g
\]
are the same as the fibers of the Abel--Jacobi map. In the range $g+1\leq n \leq 2g-2$, we always have the existence of points $p\in J$ such that $\mu_{n}^{-1}(p)=\C P^m$ with $m>n-g$. More precisely, we have that
\[
m=n-g+h^{0}(M_g, K_{M_g}-D),
\]
where $D$ is an effective \emph{special} divisor so that $h^{0}(M_g, K_{M_g}-D)>0$ and $\deg(D)=n$; for more details, see~\cite[Page 245]{GH}. This contradicts the splitting in~\eqref{cheeger}. The proof is now complete.
\end{proof}

In the range $n\geq 2g-1$, the obstruction to non-negative Ricci curvature is more subtle. It relies on Proposition~\ref{no product} and another structure result on the fundamental group of non-positively curved manifolds due to Cheeger and Gromoll.

\begin{prop}\label{true?}
For $n \geq 2g-1$ and $g>1$, the smooth manifold $SP^n(M_g)$ cannot support a Riemannian metric of non-negative Ricci curvature.
\end{prop}

\begin{proof}
We proceed by contradiction. Suppose that $SP^n(M_g)$ admits a Riemannian metric of non-negative Ricci curvature. Then by the splitting theorem of Cheeger and Gromoll~\cite[Theorem 3]{CG71}, we have that the universal Riemannian cover $\rwt{SP^n(M_{g})}$ splits \emph{isometrically} as
\begin{equation}\notag
\rwt{SP^n(M_g)}= \C P^{n-g}\times \R^{2g}.
\end{equation}
By~\cite[Theorem 9.2]{CG72}, there exists a finite regular cover
\[
\varphi: M\to SP^n(M_{g}),
\]
where $M$ is \emph{diffeomorphic} to the product $\C P^{n-g}\times T^{2g}$. For $g>1$, this is in contradiction with Proposition~\ref{no product}
\end{proof}
We conclude this section by showing that symmetric products of surfaces with $n\geq g$ cannot support K\"ahler metrics with non-positive holomorphic sectional curvature. This result, in the case $n>g$, appeared first in~\cite[Proposition 3.2]{Biswas}. Here, we provide a more elementary proof that also generalizes to the case $n=g$.

\begin{prop}\label{afterbiswas}
For $n\geq g$, the smooth manifold $SP^n(M_g)$ cannot support a Riemannian metric of non-positive holomorphic sectional curvature.
\end{prop}

\begin{proof}
For $n>g$, $SP^n(M_g)$ contains smooth rational curves. Indeed, the generic fiber of the Abel--Jacobi map
\[
\mu_n: SP^n(M_g) \to J=\C^g/\Lambda
\]
is biholomorphic to $\C P^{n-g}$. Let us assume that there is a K\"ahler metric on $SP^n(M_{g})$ with non-positive holomorphic sectional curvature. Since in a K\"ahler manifold, the holomorphic sectional curvature can only decrease along its complex submanifolds (\emph{cf}. Equation \eqref{secondF} in the proof of Proposition~\ref{it can}), the Gauss--Bonnet theorem provides a contradiction under the hypothesis of the existence of rational curves. This is because the holomorphic sectional curvature coincides with the sectional curvature for Riemann surfaces. This concludes the proof in the range $n>g$.

For $n=g$, recall that the Abel--Jacobi map
\[
\mu_n: SP^n(M_n) \to J=\C^n/\Lambda
\]
is a birational morphism with non-empty exceptional locus, see Theorem~\ref{Jacobi}. The exceptional fibers are positive-dimensional complex projective spaces.  More precisely, they are  biholomorphic to $\C P^k$'s, where $k=h^{0}(M_n, K_{M_n}-D)>0$ for an effective special divisor $D$ with $\deg(D)=n$; for more details, see~\cite[Page 245]{GH}. Thus, $SP^n(M_{n})$ contains at least one smooth rational curve, and we can argue exactly as in the case $n>g$. The proof is now complete.
\end{proof}

\section{Essentiality, LS-category, and topological complexity}\label{lstc}

In this section, we study some homotopy invariants of spaces, namely the LS-category and topological complexity, and determine their values for the symmetric products of curves.

\subsection{Essentiality}
We recall that a closed manifold $M$ of dimension $n$ is called {\em essential}~\cite{Gr1},~\cite{Gr3} if a map $u:M\to B\pi_1(M)$ that classifies the universal cover cannot be deformed to the $(n-1)$-skeleton of the CW complex $B\pi_1(M)$. 
Some obvious examples of essential manifolds include aspherical manifolds. 

It is known~\cite{Ba},~\cite{BD1} that an orientable $n$-manifold $M$ is essential if and only if $u_*([M])\ne 0$ in $H_n(B\pi_1(M))$, where $[M]$ is the fundamental class of $M$.

An orientable $n$-manifold $M$ is called {\em rationally essential} if $u_*([M])\ne 0$ in $H_n(B\pi_1(M);\mathbb Q)$.

\begin{prop}\label{essential}
For $n\le g$, the manifolds $SP^n(M_g)$ are rationally essential.
\end{prop}
 
\begin{proof}
By Corollary~\ref{mu}, the Abel--Jacobi map $\mu_n:SP^n(M_g)\to J$ induces an isomorphism of the fundamental groups and so, it is a classifying map for $SP^n(M_g)$.
By Proposition~\ref{2g-torus} (1), we have $\alpha=a_1^*b_1^*\cdots a_n^*b_n^*\ne 0$ in $H^*(SP^n(M_g)$. 
Then its Poincar\'e dual is the homology class $PD(\alpha)=[SP^n(M_g)]\frown\alpha\ne 0$.
The homomorphism $(\mu_n)_*$ is an isomorphism of 0-dimensional homology groups and so it takes 
 $[SP^n(M_g)]\frown\alpha$ to the non-zero element 
\[
\mu_*\left[SP^n(M_g)\right]\frown(a_1^*b_1^*\cdots a_n^*b_n^*).
\]
Therefore, the homomorphism $(\mu_n)_*$ takes
the fundamental class $[SP^n(M_g)]$ to a non-zero element.
Since the groups $H_*(J)$ are torsion-free, this also holds true for rational coefficients.
\end{proof}

The K\"unneth Formula implies the following.

\begin{cor}\label{essalldim}
For $n\le g$, the manifolds $SP^n(M_g)\times S^1$ are rationally essential.
\end{cor}

For $n>g$, the manifolds $SP^n(M_g)$ are inessential by dimensional reasons.

\subsection{LS-category}
We recall the definition of the following classical numerical invariant,~\cite{LS}.

\begin{defn}
Given a CW complex $X$, the \emph{Lusternik--Schnirelmann category} (LS-category) of $X$, denoted $\cat(X)$, is the smallest integer $n\ge 0$ such that there is a covering $\{U_i\}$ of $X$ by $n+1$ open sets each of which is contractible in $X$.
\end{defn}

The LS-category is a homotopy invariant that was introduced as a lower bound to the number of critical points of differentiable real-valued functions on a smooth manifold; see~\cite{CLOT} for a detailed survey on LS-category.

We recall that for any ring $R$, the cup-length of the ring $H^*(X;R)$ is a lower bound to $\cat(X)$,~\cite[Proposition 1.5]{CLOT}.
\vspace{2mm}

\begin{prop}[\protect{\cite{KR}}]~\label{KR}
For a closed $n$-manifold $M$, $\cat(M)=\dim(M)=n$ if and only if $M$ is essential.
\end{prop}

\begin{thm}\label{lscat}
For the LS-category of $SP^n(M_g)$, we have that
\[
\cat \left(SP^n(M_g)\right)=\begin{cases}
2n & \text{ if } n\le g \\
n+g & \text{ if } n >g.
\end{cases}
\]
\end{thm}
\begin{proof}
For $n\le g$, we use Theorem~\ref{essential} and Proposition~\ref{KR} to get the equality $\cat (SP^n(M_g))=2n$.

For $n>g$, we note that the inequality 
\[
\cat (X)\le\frac{\dim(X)+\cd\left(\pi_1(X)\right)}{2}
\]
from ~\cite{Dr4} for $X=SP^n(M_g)$ gives $\cat (SP^n(M_g)) \le n+g$. We also note that by Proposition~\ref{2g-torus} (2), $a_1^*b_1^*\cdots  a_g^*b_g^*(c^*)^{n-g}\ne 0$ in $H^*(SP^n(M_g))$. Hence, the cup-length of $SP^n(M_g)$ gives $n+g\le \cat (SP^n(M_g))$.
\end{proof}

\subsection{Topological complexity}
The following homotopy invariant was introduced by Farber in~\cite{Far1} in his study of the \emph{motion planning problem} in topological robotics.

\begin{defn}
Given a CW complex $X$, the \emph{topological complexity} of $X$, denoted $\TC(X)$, is the smallest integer $n\ge 0$ such that there is a covering $\{V_i\}$ of $X\times X$ by $n+1$ open sets over each of which there exists a continuous map $s_i:V_i\to P(X)$ such that $s_i(x,y)(0)=x$ and $s_i(x,y)(1)=y$ for each $(x,y)\in V_i$.
\end{defn}

Note that $\TC(X)=0$ if and only if $X$ is contractible,~\cite{Far1}.

For the applications of $\TC$ to motion planning and its computation on several classes of finite CW complexes, we refer to~\cite{Far1} and~\cite[Chapter 4]{Far2}.

Let $\Delta:X\to X\times X$ be the diagonal map. We recall from~\cite{Far2} that for any ring $R$, the \emph{$R$-zero-divisor cup-length} of $X$, denoted $zc\ell_R(X)$, is defined as the cup-length of the ideal $\text{Ker}(\Delta^*:H^*(X\times X;R)\to H^*(X;R))$.

Farber provided the following useful bounds to $\TC(X)$.

\begin{thm}[\protect{\cite{Far1}}]\label{farthm}
For any finite CW complex $X$ and ring $R$, we have that $zc\ell_R(X)\le\TC(X)\le\cat(X\times X)\le 2\cat(X)$.
\end{thm}

Using this theorem and our computation of $\cat(SP^n(M_g))$ from the previous section, we completely determine $\TC(SP^n(M_g))$ for each $n$ and $g$.

\begin{thm}\label{tcthm}
For $n \ge 2$, we have for the topological complexity of $SP^n(M_g)$ that
\[
\TC \left(SP^n(M_g)\right)=2\cat\left(SP^n(M_g)\right)=\begin{cases}
4n & \text{ if } n\le g \\
2n+2g & \text{ if } n >g.
\end{cases}
\]
\end{thm}

\begin{proof}
Recall from Section~\ref{cohomo} that for $1\le i \le g$, we have rational cohomology classes $a_i^*,b_i^*\in H^1(SP^n(M_g))$ and $c^*\in H^2(SP^n(M_g))$. We define the following non-zero rational cohomology classes in $H^*(SP^n(M_g)\times SP^n(M_g))$:
\begin{itemize}
\item $\ov{a_i^*}=a_i^*\otimes 1-1\otimes a_i^*$ and $\ov{b_i^*}=b_i^*\otimes 1-1\otimes b_i^*$ for each $i$; 
\item $\ov{c^*}=c^*\otimes 1-1\otimes c^*$.
\end{itemize}
It is easy to check that $\ov{a_i^*},\ov{b_i^*},\ov{c^*}\in\text{Ker}(\Delta^*)$; see, for example,~\cite[Example 4.38]{Far2}. Using Macdonald's relations $(a_i^*)^2=(b_i^*)^2=0$, we get that $(\ov{a_i^*})^2=2(a_i^*\otimes a_i^*)\ne 0$ and $(\ov{b_i^*})^2=2(b_i^*\otimes b_i^*)\ne 0$. 

When $n\le g$, we have from Proposition~\ref{2g-torus} (1) that $a_1^*b_1^*\cdots a_n^*b_n^*\ne 0$. Hence,
\[
(\ov{a_1^*})^2(\ov{b_1^*})^2\cdots (\ov{a_n^*})^2(\ov{b_n^*})^2 = 2^{2n} (a_1^*\otimes a_1^*) (b_1^*\otimes b_1^*) \cdots (a_n^*\otimes a_n^*) (b_n^*\otimes b_n^*)
\]
\[
= 2^{2n} \left(a_1^*b_1^*\cdots a_n^* b_n^*\right)\otimes \left(a_1^*b_1^*\cdots a_n^* b_n^*\right) \ne 0.
\]
Therefore, we get $4n\le zc\ell_{\Q}(SP^n(M_g))\le \TC(SP^n(M_g))\le 2\cat(SP^n(M_g))=4n$ for $n\le g$ in view of Theorem~\ref{farthm}.

When $n>g$, we have from Proposition~\ref{2g-torus} (2) that $a_1^*b_1^*\cdots a_g^*b_g^*(c^*)^{n-g}\ne 0$. Hence, as before, we obtain
\[
(\ov{a_1^*})^2(\ov{b_1^*})^2\cdots (\ov{a_g^*})^2(\ov{b_g^*})^2 = 2^{2g} \left(a_1^*b_1^*\cdots a_g^* b_g^*\right)\otimes \left(a_1^*b_1^*\cdots a_g^* b_g^*\right) \ne 0.
\]
Furthermore, the cup product $(\ov{c^*})^{2n-2g}$ contains the term
\[
(-1)^{n-g}{2n-2g \choose n-g} (c^*)^{n-g}\otimes (c^*)^{n-g} \ne 0.
\]
It can be deduced from Macdonald's relations (Theorem~\ref{Mc}), and independently from Theorem~\ref{lscat}, that $a_1^*b_1^*\cdots a_g^*b_g^*(c^*)^{n-g+k}=0$ for each $k\ge 1$. Thus, the cup product $(\ov{a_1^*})^2(\ov{b_1^*})^2\cdots (\ov{a_g^*})^2(\ov{b_g^*})^2(\ov{c^*})^{2n-2g}$ is equal to the term
\[
(-1)^{n-g}\hspace{0.5mm}2^{2g}{2n-2g \choose n-g} \left(a_1^*b_1^*\cdots a_g^* b_g^*(c^*)^{n-g}\right)\otimes \left(a_1^*b_1^*\cdots a_g^* b_g^*(c^*)^{n-g}\right) \ne 0.
\]
Hence, $2n+2g\le zc\ell_{\Q}(SP^n(M_g))\le \TC(SP^n(M_g))\le 2\cat(SP^n(M_g))=2n+2g$ for $n>g$ in view of Theorem~\ref{farthm}. This completes the proof.
\end{proof}

\begin{rem}
In~\cite{DJ} (and independently in~\cite{KW}), a probabilistic version of the  LS-category and the topological complexity, denoted by $\dcat$ and $\dTC$, respectively, was introduced. It turns out that for symmetric products $SP^n(M_g)$, these probabilistic invariants agree with their classical counterparts,~\cite{Ja}.
\end{rem}

\section{Macroscopic dimensions of universal Riemannian covers}\label{Cmacroscopic}

In this section, we study the interactions between curvatures and two distinct notions of macroscopic dimension. In particular, we address the question of when these two macroscopic dimensions agree for the universal covers of closed smooth manifolds. 

\begin{defn}
Given metric spaces $X$ and $Y$, we say that $f:X\to Y$ is \emph{uniformly cobounded} if there exists $C>0$ such that $\text{diam}(f^{-1}(y))<C$ for all $y\in Y$.
\end{defn}

The macroscopic dimension, $\dim_{mc}$, was defined by Gromov as follows.

\begin{defn}[\cite{Gr2}]\label{dmc}
For a Riemannian manifold $X$, its \emph{macroscopic dimension}, denoted $\dim_{mc}X$, is the smallest integer $n\ge 0$ such that there is a uniformly cobounded continuous map $f:X\to K^n$ to an $n$-dimensional simplicial complex $K^n$. 
\end{defn}

A modification of macroscopic dimension, denoted $\dim_{MC}$, was introduced in~\cite{Dr2} as follows.

\begin{defn}\label{dMC}
For a Riemannian manifold $X$, $\dim_{MC}(X)\le n$ if there is a Lipschitz uniformly cobounded proper map $g:X\to K^n$ to an $n$-dimensional simplicial complex given a uniform metric.
\end{defn}

We have a chain of inequalities:
\begin{equation}\label{basicinequality}
\dim_{mc}(X)\le\dim_{MC}(X)\le \dim(X).
\end{equation}

\begin{remark}
It was proven in~\cite[Proposition 2.1]{Dr3} that for any proper metric space $X$, the inequality $\dim_{mc}X\le k$ implies the existence of a proper continuous map $f:X\to K$ to a locally finite $k$-dimensional simplicial complex $K$.  Thus, the \emph{proper condition} can be added to the map $f$ in  Definition~\ref{dmc} when we talk about the macroscopic dimensions of universal covers of closed manifolds. Here, for a closed manifold $M$ supplied with a geodesic metric, we consider the lifted metric on its universal cover $\Wi M$.
\end{remark}

We now address the following tantalizing question.

\begin{question}\label{impques}
Let $M$ be a closed Riemannian manifold and let $\Wi{M}$ be the universal Riemannian cover of $M$. Under which conditions on $M$ do we have the equality $\dim_{mc}(\Wi{M})=\dim_{MC}(\Wi{M})$?
\end{question}

In particular, we will be uniquely concerned with the macroscopic dimensions of universal Riemannian covers of closed Riemannian manifolds.

\subsection{Relation with curvature}
In this section, we will answer Question~\ref{impques} by finding curvature conditions on $M$ which ensure $\dim_{mc}\Wi M=\dim_{MC}\Wi M$. We show that this holds true for large classes of Riemannian manifolds.

Let us start by recalling the following result.


\begin{thm}[\protect{\cite[Theorem 5.4]{Dr3}}]\label{small}
Let $M$ be a closed orientable $n$-manifold $M$ with fundamental group $\pi=\pi_1(M)$, a classifying map $u_M:M\to B\pi$, and a lift $\Wi u_M:\Wi M\to E\pi$ of $u_M$ to the universal covers. Then the following statements are equivalent.
\begin{enumerate}
\item $\dim_{mc}\Wi M<n$.
\item $(\Wi u_M)_*([\Wi M])=0$ in $H^{lf}_n(E\pi;\mathbb Z)$, where $[\Wi M]\in H^{lf}_n(\Wi M;\mathbb Z)$ is the fundamental class of $\Wi M$.
\end{enumerate}
\end{thm}

Here, for a CW complex $X$, $H_*^{lf}(X;\Z)$ denotes the integral homology of $X$ defined by locally finite chains.

\begin{cor}\label{nondim}
If $M$ is a closed orientable aspherical $n$-manifold, then 
\[
\dim_{mc}\Wi M=\dim_{MC}\Wi M=n.
\]
\end{cor}
More generally, we have the following. 
\begin{prop}\label{Dranishnikov}
Let $M$ be an orientable closed $n$-manifold. Suppose that a degree one map between $f:N\to M$ induces an isomorphism of the fundamental groups and $\dim_{mc}\Wi M=n$. Then $\dim_{mc}\Wi N=\dim_{MC}\Wi N=n$.
\end{prop}
\begin{proof}
Let $u_M:M\to B\pi$ be a classifying map for $M$. Clearly, the composition $u_N=u_M\circ f$ is a classifying map for $N$.
Then in view of the equality $\dim_{mc}\Wi M=n$, we obtain $\Wi{u_M}_*([\Wi M])\ne 0$ by Theorem~\ref{small}. Since $\Wi f_*:H_n^{lf}(\Wi N)\to H^{lf}_n(\Wi M)$ is an isomorphism of groups isomorphic to $\mathbb Z$, we obtain $$\Wi{u_N}_*([\Wi N])=(\Wi u_M)_*(\Wi f_*([\Wi N]))\ne 0.$$ Hence, $\dim_{mc}\Wi N=n$. Since $\dim_{mc}\Wi N\leq \dim_{MC}\Wi N\leq \dim \Wi N=n$, we obtain the equality $\dim_{MC}\Wi N=n$ as well.
\end{proof}

\begin{cor}\label{newproof}
    For each $n\ge 1$,
    \[
\dim_{mc}\rwt{SP^n(M_n)}=\dim_{MC}\rwt{SP^n(M_n)}=2n.
    \]
\end{cor}
\begin{proof}
    By Corollary~\ref{nondim}, $\dim_{mc} \Wi{T^{2n}} = 2n$. Therefore, this follows from Proposition~\ref{Dranishnikov} in view of the degree one Abel--Jacobi map $\mu_n:SP^n(M_n)\to T^{2n}$. 
\end{proof}

\begin{remark}\label{carthad}
By the Cartan--Hadamard theorem (\emph{e.g.},~\cite[Theorem 6.2.2]{Petersen}), we know that if $M$ admits a metric with non-positive sectional curvature, then $\Wi{M}$ is diffeomorphic to $\R^n$ and therefore contractible. By Corollary~\ref{nondim}, we conclude that the  $\dim_{mc}\Wi{M} = \dim_{MC}\Wi{M}$ for such spaces. Interestingly, in Section~\ref{macdimsec}, we will show that this fact  does \emph{not} generalize to K\"ahler manifolds with non-positive holomorphic sectional curvature. Indeed, we will provide examples of closed K\"ahler manifolds $(M^k, g_\omega)$ of complex dimension $k$ for any $k\geq 2$ having non-positive holomorphic sectional curvature such that 
\[
\dim_{mc}\Wi{M} < \dim_{MC}\Wi{M}.
\]
\end{remark}

On the opposite spectrum of curvature, we observe the following.

\begin{prop}\label{PRdim}
If $M$ is a closed $n$-manifold that admits a Riemannian metric of non-negative Ricci curvature, then $\dim_{mc}\Wi{M} = \dim_{MC}\Wi{M}\leq n$. Moreover, the inequality is saturated if and only if $M$ is a flat Riemannian $n$-manifold.
\end{prop}

\begin{proof}
It follows from the splitting theorem of Cheeger and Gromoll~\cite[Theorem 3]{CG71} that the universal Riemannian cover $\Wi{M}$ splits \emph{isometrically} as
\[
\Wi{M}= N^{n-l}\times \R^l,
\]
where $N^{n-l}$ is a simply connected compact manifold with non-negative Ricci curvature. The projection onto the Euclidean factor is a uniformly cobounded continuous, Lipschitz proper map, so that $\dim_{mc}\le\dim_{MC}\Wi{M}\leq l$. If we assume that $\dim_{mc}\Wi{M}< l$ (resp. $\dim_{MC}\Wi{M}< l$), then there is a continuous (resp. Lipschitz) uniformly cobounded proper map
\[
\varphi: \Wi{M}\to K^s
\]
to an $s$-dimensional simplicial complex $K^{s}$ with $s<l$. The restriction of such a map to any of the embedded totally geodesic flat Euclidean spaces $n\times \R^l$, where $n\in N^{n-l}$, gives that $\dim_{mc}\R^l\leq s<l$. This contradicts Corollary~\ref{nondim}. Hence, we conclude that
\[
\dim_{mc}\Wi{M} = \dim_{MC}\Wi{M}=l.
\]
Moreover, it is easy to see that $l=n$ if and only if $M$ is a flat $n$-manifold.
\end{proof}

Therefore, if $\dim_{mc}\neq \dim_{MC}$ on the universal cover of a closed Riemannian manifold $M$, then $M$ cannot support a Riemannian metric of non-negative Ricci curvature. 

\subsection{Behavior on connected sums}\label{on connected sums}
In this section, we investigate the behavior of macroscopic dimensions on universal Riemannian covers of connected sums of Riemannian manifolds to find out more classes of Riemannian manifolds where the two notions of macroscopic dimensions agree.

We begin by noting that the remark after the proof of~\cite[Theorem 2.2]{Dr3} spells out the following theorem.

\begin{thm}[\protect{\cite{Dr3}}]\label{0}
Let $M$ be a finite CW complex with a geometrically finite fundamental group $\pi=\pi_1(M)$, a classifying map $u_M:M\to B\pi$, and a lift $\Wi u_M:\Wi M\to E\pi$ of $u_M$ to the universal covers. Then the following statements are equivalent for any $k\ge 0$.
\begin{enumerate}
\item $\dim_{mc}\Wi M\le k$.
    
\item There is a continuous map $f:\Wi M\to E\pi^{(k)}$ with $\textup{dist}(f,\Wi u_M)<\infty$. 
\item There is a bounded homotopy $H:\Wi M\times[0,1]\to E\pi$ of $\Wi u_M:\Wi{M}\to E\pi$ to a map $f:\Wi M\to E\pi^{(k)}$. 
\end{enumerate}
\end{thm}

We recall that a group $\pi$ is called {\em geometrically finite} if it admits a finite classifying complex $B\pi$. 

\begin{remark}
If $\pi$ is just finitely presented, we consider a locally finite complex $B\pi$ and a proper metric on it. Then Theorem~\ref{0} holds when one adds the following condition to (2) and (3): the space $pf(\Wi M)$ is contained in a compact subset of $B\pi$, where $p:E\pi \to B\pi$ is the universal cover of $B\pi$.
\end{remark}

Let $M$ and $N$ be closed $n$-manifolds for $n\ge 3$. Let $\Gamma=\pi_1(M)$ and $\Lambda=\pi_1(N)$. Then $\pi_1(M\# N)=\Gamma\ast\Lambda$ and $B(\Gamma\ast\Lambda)=B\Gamma\bigvee B\Lambda$. Let $y_0$ be the wedge point and $x_0\in E(\Gamma\ast\Lambda)$ be its lift. Therefore, $E(\Gamma\ast\Lambda)$ is the union of disjoint copies of $E\Gamma$ indexed by $\Lambda$ and disjoint copies of $E\Lambda$ indexed by $\Gamma$, with the intersection of $E\Gamma$ with $E\Lambda$ being either empty or a singleton point $w(x_0)$, where $w\in\Gamma\ast\Lambda$. Moreover, the nerve of the cover of $E(\Gamma\ast\Lambda)$ by $E\Gamma$ and $E\Lambda$ is a tree.

\begin{thm}\label{top}
For a given $n$-manifold $M$ with $\dim_{mc}\Wi M=n\ge 3$, there is the equality
$$
\dim_{mc}\Wi{M\# N}=\dim_{MC}\Wi{M\# N}=n
$$ 
for any manifold $N$.
\end{thm}

\begin{proof}
Let $S$ be a separating $(n-1)$-sphere in $M\# N$. The collapsing map $$q:M\# N\to (M\# N)/S=M\bigvee N$$ followed by the wedge map
\[
u_M\bigvee u_N:M\bigvee N\to B\Gamma\bigvee B\Lambda
\]
of the classifying maps $u_M$ and $u_N$ is a classifying map $u:M\# N\to B(\Gamma\ast\Lambda)$. Let $p:E(\Gamma\ast\Lambda)\to B\Gamma\bigvee B\Lambda$ be the universal cover. Let $\{y_0\}=B\Gamma\cap B\Lambda$ and $p(x_0)=y_0$. Let $E\Gamma_0$ be a copy of $E\Gamma$ that contains $x_0$. Fix a section $S'\subset\Wi M^0$ of $S$, where $M^0=M\setminus \Int B^n$ and $S=\partial B^n$. We consider a lift 
$$
\Wi u:\Wi{M\# N}\to E(\Gamma\ast\Lambda)
$$ 
of $u$ that takes $S'$ to $x_0$.
We may assume that $u$ is 1-Lipschitz, in which case $\Wi u$ is 1-Lipschitz as well.

Let us assume that $\dim_{mc}\Wi{M\# N}\le n-1$. Then by Theorem~\ref{0}, there is a map 
$$
f:\Wi{M\# N}\to E(\Gamma\ast\Lambda)^{(n-1)}
$$ 
with $\text{dist}(\Wi u, f)<b$ for some $b$. Let $N_a(E\Gamma_0)$ be a closed $a$-neighborhood of $E\Gamma_0$ for some $a>b$ such that $\partial N_a(E\Gamma_0)\cap (\Gamma\ast\Lambda)(x_0)=\emptyset$, where $(\Gamma\ast\Lambda)(x_0)$ is the orbit of $\{x_0\}$ under the action of $\Gamma \ast \Lambda$. Let 
$$
W=(\Wi u)^{-1}(N_a(E\Gamma_0)).
$$ 
Then we have that $f(\partial W)\cap E\Gamma_0=\emptyset$. There is a manifold $V$ with boundary $\pa V$ such that $V\subset W\subset \Wi{M\# N}$ and $f(\partial V)\cap E\Gamma_0=\emptyset.$ We may assume that $\partial V=\coprod S_i$, where each $S_i$ lives in a copy of $\Wi M^0$ or $\Wi N^0$. Each manifold $S_i$ bounds a compact manifold $X_i$  either in $\Wi M$ or $\Wi N$.  We consider an open manifold without boundary $X=V\cup(\coprod X_i)$. It is easy to see that $X$ admits a proper map of degree $1$, say $g:X\to\Wi M$. Next, let $r:E(\Gamma\ast\Lambda)\to E\Gamma_0$ be the natural retraction which maps the complement of $E\Gamma_0$ to $\Gamma x_0$. Then the composition
$rf|_V:V\to E\Gamma$ has image in $E\Gamma_0^{(n-1)}$. Since $f(\partial V)\cap E\Gamma_0=\emptyset$, the map $rf|_V$ extends to a map $\hat f:X\to E\Gamma^{(n-1)}_0$. Thus, we obtain two proper maps $\Wi u_M \circ g:X\to E\Gamma$ and $\hat f:X\to E\Gamma$ which are in a finite distance, and hence, proper homotopic. By Theorem~\ref{small}, for the first map we have $(\tilde u_M)_*(g_*([X]))\ne 0$, whereas for the second map we have $\hat f_*([X])=0$. This is a contradiction.
\end{proof}

Another interesting feature of both macroscopic dimensions is that they coincide in low dimensions.

\begin{thm}\label{lowdim}
If $M$ is a closed $2$- or $3$-manifold, then $\dim_{mc}\Wi{M} = \dim_{MC}\Wi{M}$.
\end{thm}

\begin{proof}
In the $2$-dimensional case, because of the classification theorem for Riemann surfaces, it suffices to consider surfaces of genus $g\geq 0$ denoted by $M_g$. When $g=0$, we have $M_0=S^2$, and therefore
\[
\dim_{mc}\Wi{M_0} = \dim_{MC}\Wi{M_0}=0.
\]
For $g\geq 1$, $M_g$ can be equipped with either a flat or hyperbolic Riemannian metric, and by Corollary~\ref{nondim} (see also Remark~\ref{carthad}), we have
\[
\dim_{mc}\Wi{M_g} = \dim_{MC}\Wi{M_g}=2.
\]
In the $3$-dimensional case, we follow the same philosophy by going through all possible closed orientable $3$-manifolds in the list provided as a by-product of Perelman's proof of Thurston's and Poincar\'e's conjectures. We refer to~\cite{Lott} for the relevant background on the geometry and topology of closed $3$-manifolds and the details of Perelman's proof. Now, if $M$ is a closed $3$-manifold with finite $\pi_1$, then $\Wi{M}=S^3$, so that 
\[
\dim_{mc}\Wi{M} = \dim_{MC}\Wi{M}=0.
\]
If $\pi_1$ is infinite, $M$ is either aspherical, has an aspherical component in its prime decomposition as a connected sum, or is a finite connected sum of $S^2\times S^1$'s and spherical space forms, i.e., we have 
\begin{equation}\label{positive3}
M=(S^2\times S^1)\#\cdots\#(S^2\times S^1)\# S^3/\Lambda_1\#\cdots\# S^3/\Lambda_j
\end{equation}
in the last case. If $M$ is aspherical, we know from Corollary~\ref{nondim} that
\[
\dim_{mc}\Wi{M}=\dim_{MC}\Wi{M}=3.
\]
If $M$ has an aspherical component in its prime decomposition, then we again have $\dim_{mc}\Wi{M}=\dim_{MC}\Wi M=3$ because of Theorem~\ref{top}. The remaining $3$-manifolds are as in \eqref{positive3}. For such a $3$-manifold $M$, it follows from~\cite[Corollary 10.11]{GL} that one can construct a continuous \emph{distance non-increasing} map from $M$ to a metric graph. By lifting this map to the universal covers, we obtain a continuous uniformly cobounded distance non-increasing proper map from $\Wi{M}$ to a metric tree. Since distance non-increasing maps are Lipschitz, we conclude that both macroscopic dimensions are at most one. Since $\Wi{M}$ is not compact, none of these macroscopic dimensions can be zero, and we conclude that the $3$-manifolds as in \eqref{positive3} with infinite $\pi_1$ satisfy the equality
\[
\dim_{mc}\Wi{M}=\dim_{MC}\Wi{M}=1.
\]
This concludes the proof.
\end{proof}

\begin{remark}
The equality $\dim_{mc}\Wi M=\dim_{MC}\Wi M=3$ for any closed Riemannian $3$-manifold $M$ having an aspherical summand in its prime decomposition can also be proved directly. Since $\pi_1(M)$ is infinite, we must have $1\leq\dim_{mc}\Wi{M}\leq 3$. It follows from~\cite{Bolotov3} that $\dim_{mc}\Wi{M}\neq 2$. Since $\pi_2=0$ for any $1$-dimensional simplicial complex, if we assume $\dim_{mc}\Wi{M}=1$, we can easily construct a uniformly cobounded continuous proper map from the universal cover of the aspherical component of $M$ to a $1$-dimensional simplicial complex. This contradicts Corollary~\ref{nondim}, and we obtain $\dim_{mc}\Wi{M}=\dim_{MC}\Wi{M}=3$.
\end{remark}

The connected sum formula for LS-category 
\[
\cat(M\# N)=\max\{\cat(M),\cat (N)\}
\]
was proven in~\cite{DS}. This and Theorem~\ref{cattheo} imply that
\begin{equation}\label{obvimpl}
    \dim_{mc}\Wi{M\# N}\le \max\{\cat(M),\cat(N)\}
\end{equation}
for closed $n$-manifolds $M$ and $N$. Note also that Theorem~\ref{top} gives the inequality 
\[
\dim_{mc}\Wi{M\# N}\ge\max\{\dim_{mc} \Wi M,\dim_{mc} \Wi N\}
\]
whenever the right-hand side equals $n$. 

For $n>2$, we prove the above inequality in the other direction in full generality, thereby improving the upper bound 
from~\eqref{obvimpl}.

\begin{prop}\label{5}
Let $n>2$. Then for $n$-manifolds $M$ and $N$,
\[
\dim_{mc}\Wi{M\# N}\le \max\{\dim_{mc} \Wi M,\dim_{mc}\Wi N\}.
\]
\end{prop}

\begin{proof}
We use the notations from the paragraph preceding Theorem~\ref{top}. Suppose that $\dim_{mc}\Wi M\le\dim_{mc}\Wi N= k$. Let $B_M\subset M$ and $B_N\subset N$ be the $n$-balls such that they have the common boundary $S$ in $M\# N$. We may assume that $u_M(B_M)=y_0\in B\Gamma$ and $u_N(B_N)=y_0\in B\Lambda$. Recall that $x_0$ is a lift of $y_0$ in $E\Gamma$ (resp. $E\Lambda$) via a path $\gamma$ (resp. $\lambda$). Let $B_\gamma$ denote the translation of the ball $B_M$ by $\gamma$ in the universal cover $\Wi M$. Then $\Wi u_M(B_\gamma)=\gamma x_0$.
Due to Theorem~\ref{0}, there is a map $f_M:\Wi M\to E\Gamma^{(k)}$ in a finite distance to $\Wi u_M$. Hence, we may assume that $f_M(B_\gamma)=\gamma x_0$. Similarly, we can find a map $f_N:\Wi N\to E\Lambda^{(k)}$ that takes $B_\lambda$ to $\lambda x_0$. The union of these maps defines a map  
\[
\bar f:\Wi{M\cup_BN}\to E(\Gamma\ast\Lambda)^{(k)}.
\]
Here, $M\cup_BN$ is formed by identifying $B_M$ and $B_N$ in $M\cup N$. The restriction of $\bar f$ to $\Wi{M\#N}\subset \Wi{M\cup_BN}$ is a map that demonstrates the inequality $\dim_{mc}\Wi{M\# N}\le k$.
\end{proof}

\begin{rem}
We note that in dimension $n=2$, the proofs of Theorem~\ref{top} and Proposition~\ref{5} do not work because for closed $2$-manifolds $M$ and $N$, the fundamental groups $\pi_1(M\# N)$ and $\pi_1(M)\ast \pi_1(N)$ need not be isomorphic. In fact, Proposition~\ref{5} is not true in dimension $2$. Indeed, 
\[
2=\dim_{mc}(\rwt{\mathbb RP^2\#\mathbb RP^2})>\dim_{mc}\rwt{\mathbb RP^2}=0.
\]
\end{rem}


We conclude this section by showing another interesting property of macroscopic dimensions.

\begin{thm}\label{cattheo}
For a finite CW complex $Y$,
$$
\dim_{mc}\Wi Y\le \dim_{MC}\Wi Y\le \cat (Y).$$
\end{thm}

\begin{proof}
Let $\cat(Y)= k$, and let $\mathcal U=\{U_0,\ldots, U_k\}$ be a cover of $Y$ by open sets $U_i$ contractible in $Y$. We may assume that all $U_i$ are connected. Let $p_Y:\Wi Y\to Y$ denote the universal cover. Note that for each $i$, the preimage $p^{-1}(U_i)=\coprod_{\gamma\in\pi}C_\gamma^i$ is the disjoint union of components indexed by elements of the fundamental group $\pi=\pi_1(Y)$. Furthermore, all components are isometric and bounded. This follows from a lifting of a deformation of $U_i$ in $Y$ to a point.
Let us now consider the following open cover of $\Wi Y$:
\[
\mathcal W=\left\{C^i_\gamma \hspace{1mm} \middle | \hspace{1mm} i=0,\ldots, k, \hspace{2mm}  \gamma\in\pi\right\}.
\]
Let $f:\Wi Y\to N=N(\mathcal W)$ be the projection to the nerve. Clearly, $\dim(N)\le k$ and $f$ is a uniformly cobounded Lipschitz map. Hence, $\dim_{MC}\Wi Y\le k$. 
\end{proof}

In Section~\ref{lstc} and in this section, we considered the following numerical topological invariants for a closed manifold $M$: the LS-category $\cat(M)$, the topological complexity $\TC(M)$, and the macroscopic dimensions of the universal cover $\dim_{mc}\Wi M$ and $\dim_{MC} \Wi M$. Yuli~Rudyak proposed the following conjecture for any topological (homotopy) numerical invariant.
\vspace{2.5mm}
\begin{conjec}[Rudyak Conjecture]~\label{Rudyak} 
For a degree one map $f : M \to N$ between orientable $n$-manifolds, there is the inequality $\mathbbm n(M) \ge \mathbbm n(N) $ for any numerical homotopy invariant $\mathbbm n$.
\end{conjec}

Now the name Rudyak Conjecture is mostly associated with the LS-category, $\mathbbm n(M) =\cat(M)$,~\cite{Ru}. The conjecture is open for all four invariants listed above, though there are some partial results in the case of cat. In this paper, we are interested in this conjecture for macroscopic dimensions in the case when the degree one map is a birational map between projective varieties (see Theorem~\ref{uniruledT}).

\section{Gromov--Lawson and Gromov conjectures for K\"ahler metrics, and $SP^n(M_n)$}\label{glsection}

In this section, we draw some connections between the theories of the minimal model, positivity in complex algebraic geometry, and macroscopic dimensions. 

The first result we present provides some support to the following long-standing conjecture.

\begin{conjec}[Gromov--Lawson]\label{GLC}
A closed aspherical $n$-manifold cannot support a Riemannian metric of positive scalar curvature. 
\end{conjec}

More precisely, we prove a strengthening of Conjecture~\ref{GLC} for the scalar curvature associated with a K\"ahler metric. Finally, we apply this result to the symmetric product of curves $SP^n(M_n)$.

Recall that given a K\"ahler manifold $(M, g_{\omega})$, in a holomorphic chart, the K\"ahler metric can be expressed as  $g_\omega=\sum_{i, j}g_{i\bar{j}}dz_i\otimes dz_{\bar{j}}$. We also know that the closed $(1, 1)$-form (the Ricci form) defined locally as
\[
Ric_{g_{\omega}}:=-\sqrt{-1}\partial \bar{\partial}\log\left(\det{g_{i\bar{j}}}\right)
\]
satisfies the following cohomological identity in $H^2(M; \C)$:
\begin{equation}\label{Ricci}
[Ric_{g_{\omega}}]=2\pi c_{1}(M)=-2\pi c_{1}(K_{M}),
\end{equation}
where $c_{1}(M)$ is the first Chern class of the underlying manifold and $c_{1}(K_{M})$ is the first Chern class of the canonical line bundle $\pi: K_{M}\to M$. For these facts, we refer to~\cite[Chapter I]{GH}.

Our result asserts that a strengthening of Gromov and Lawson's Conjecture~\ref{GLC} is true if we restrict our attention to varieties that admit an aspherical smooth minimal model and are equipped with K\"ahler metrics. Along the way, we also compute the macroscopic dimension of such spaces.

\begin{thm}\label{asphericalK}
A smooth projective $n$-variety $M$ with a birational morphism onto an aspherical smooth projective $n$-variety $N$ cannot support a K\"ahler metric of positive scalar curvature. Moreover, we have $\dim_{mc}\Wi{M}=\dim_{MC}\Wi{M}=2n$.
\end{thm}

\begin{proof}
We begin assuming that $N=M$, so that $M$ is an aspherical smooth projective variety of complex dimension $n$. It is well-known that the universal cover $\Wi{M}$ does not contain any positive-dimensional complex subvariety; see, for example,~\cite[Proposition 6.7]{LMW21}. In particular, $M$ does not support any rational curves. Then, by Mori's Cone theorem~\cite[Theorem 1.24]{Mori}, we have that the canonical line bundle $K_{M}$ is nef. Nef divisors are in the closure of the ample cone, and, in particular, they are pseudo-effective. We refer to~\cite[Chapter 6]{Demailly} for the general definitions and, in particular, to ~\cite[Section 6.C]{Demailly} for the detailed description of these positive cones from both an algebraic and analytical point of view. Now, since the line bundle $K_M$ is pseudo-effective, it can be equipped with a singular Hermitian metric $h=e^{-\psi}$ whose curvature $i\Theta_h=\sqrt{-1}\partial\bar{\partial}\psi$ is a closed \emph{positive} current $T$ representing the cohomology class $2\pi c_{1}(K_{M})$; see, for example, \cite[Sections 6.A and 6.C]{Demailly}. Let $\omega$ be the K\"ahler $(1, 1)$-form associated to $g_{\omega}$ and consider the closed $(n-1, n-1)$-form $\omega^{n-1}$. Since $T$ is a $(1, 1)$-current, it is an element in the dual space of smooth (compactly supported) forms of degree $(n-1, n-1)$. Because of positivity, by testing $T$ on $\omega^{n-1}$, we have that
\begin{equation}\label{Tpositivity}
\langle T, \omega^{n-1}\rangle=\int_{M}i\Theta_h\wedge \omega^{n-1}\geq 0.
\end{equation}
On the other hand, a standard curvature computation gives the pointwise equality
\[
Ric_{g_{\omega}}\wedge \omega^{n-1}=\frac{2}{n}s_{g_{\omega}}\omega^n
\]
on $M$, where $s_{g_{\omega}}: M\to \R$ is the scalar curvature function on the K\"ahler metric $g_{\omega}$. Since the Riemannian volume element $d\mu_{g_{\omega}}$ is a positive multiple of the $(n, n)$-form $\omega^n$, we conclude that the positivity of the scalar curvature implies that
\[
\int_{M}Ric_{g_{\omega}}\wedge \omega^{n-1}>0.
\]
Because of the cohomological identity in~\eqref{Ricci}, we have
\[
\langle T, \omega^{n-1}\rangle=-\int_{M}Ric_{g_{\omega}}\wedge \omega^{n-1}<0,
\]
which contradicts the non-negativity of the current $T$ in~\eqref{Tpositivity}.  Finally, we know by Corollary~\ref{nondim} that $\dim_{mc}\Wi{M}=\dim_{MC}\Wi{M}=2n$, and the proof is complete in this particular case. 

In the general case when $M$ is not necessarily aspherical, we have a birational morphism $\pi: M\to N$ which is not the identity. The canonical line bundles of $M$ and $N$ are related by the formula
\[
K_{M}=\pi^*K_N+R,
\]
where $R$ is an effective divisor on $M$. In particular, since we observed that $K_{N}$ is nef and hence pseudo-effective, we have that $K_{M}$ is pseudo-effective as well. Following line-by-line the proof we described in the aspherical case, we conclude that $M$ cannot admit a K\"ahler metric of positive scalar curvature. Since the fundamental group is a birational invariant, the map $\pi: M\to N$ classifies the universal cover $\Wi{M}$ and $N=B\pi_1(M)$. Both $M$ and $N$ are complex manifolds and, therefore, orientable. Also, $\pi$ maps the fundamental class of $M$ to the fundamental class of $N$. This implies that $M$ is essential. By Proposition~\ref{Dranishnikov}, we deduce that
\[
\dim_{mc}\Wi{M}=\dim_{MC}\Wi{M}=\dim_{mc}\Wi{N}=\dim_{MC}\Wi{N}=2n.
\]
\end{proof}

We conclude this section by applying the machinery we developed to $SP^n(M_n)$.

\begin{cor}\label{nequalg}
For any $n\geq 2$, $SP^n(M_n)$ cannot support a K\"ahler metric of positive scalar curvature. Moreover, $SP^n(M_n)$ is essential and Corollary~\ref{newproof} is obtained.
\end{cor}
\begin{proof}
By Theorem~\ref{Jacobi} (Jacobi's theorem), the map
\[
\mu_{n}: SP^n(M_n)\to J
\]
is a birational morphism onto the $n$-dimensional Abelian variety $J$ (topologically a $2n$-torus). The conclusions are now direct consequences of Theorem~\ref{asphericalK}.
\end{proof}

Our next result addresses the following conjecture in the K\"ahler setting.

\begin{conjec}[Gromov]\label{G}
Let $(M^n, g)$ be a closed Riemannian $n$-manifold with positive scalar curvature. Then  we have $\dim_{mc}\Wi{M}\leq n-2$. 
\end{conjec}
There is a weak version of Gromov's conjecture that states in the above setting that $\dim_{mc}\Wi{M}\leq n-1$. It was stated in~\cite{Gr1} in a different language. Here we prove the weak Gromov Conjecture for smooth projective varieties that admit a K\"ahler metric with positive scalar curvature. Before stating and proving this result, we highlight the following. We get full Conjecture~\ref{G} for totally non-spin manifolds $M$  (see Definition~\ref{totnonspin}). Also, we prove it for all manifolds $M$ \emph{modulo} Rudyak’s Conjecture~\ref{Rudyak} for $\dim_{mc}$. Finally, with the aid of the Kodaira--Enriques classification, we prove an optimal statement for all complex surfaces.

We begin by recalling the following definition.

\begin{defn}\label{uniruled}
A smooth projective $n$-variety $M^n$ is said to be \emph{uniruled} if through any point $p\in M$, there exists a rational curve $C$ passing through $p$.
\end{defn}

\begin{thm}\label{uniruledT}
Let $M$ be a smooth projective $n$-variety that supports a K\"ahler metric with positive scalar curvature. We have $\dim_{mc}\Wi{M}\leq \dim_{MC}\Wi{M}\leq2n-1$. In the case $n=2$, we have $\dim_{mc}\Wi{M}=\dim_{MC}\Wi{M}\leq 2$
\end{thm}

\begin{proof}
By~\cite[Theorem 1]{Heier}, if $(M^n, g_{\omega})$ has positive scalar curvature, then $M^n$ is \emph{uniruled}. This result crucially relies on the breakthrough of Boucksom \emph{et al.}~\cite{Paun} characterizing uniruled varieties as the ones having canonical class not pseudo-effective.  Since $K_{M}$ is not pseudo-effective, by~\cite[Corollary 1.3.3]{Birkar}, we obtain that $M$ is birational to a Mori fiber space $\eta: Y\to Z$. Recall that by definition, such a space has terminal singularities and the generic fiber is Fano, say $F$. By~\cite[Theorem 1.2]{Takayama}, we obtain that $\pi_1(Y)=\pi_1(Z)$. By resolving the locus of indeterminacy of the birational map from $M$ to $Y$, we can assume that there exists $M^\prime$ smooth and birational to $M$ and a fibration
\[
\varphi: M^{\prime} \to  Z^{\prime},
\] 
over a smooth base $Z^\prime$ with $\dim_{\C}{Z^\prime}\leq n-1$ with generic fiber $F^\prime$. Moreover, by~\cite[Corollary 1.3]{Hacon}, we have that
\[
\pi_1(\varphi^{-1}(p))=0
\]  
for any $p\in Z^\prime$.
This construction yields
\[
\pi_1(M)=\pi_1(M^\prime)=\pi_1(Z)=\pi_1(Z^\prime).
\]
Since $\dim_{\R} Z^\prime\le 2n-2$, the classifying map $u^\prime:Z^\prime\to B\pi$ of its universal cover lands in the $(2n-2)$-dimensional skeleton $B\pi^{(2n-2)}$.
Then $u^\prime \phi: M^\prime\to B\pi^{(n-2)}$ is the classifying map for $\rwt{M^\prime}$, where $\phi: M^\prime \to Z^\prime$ is the obvious map. We may assume that $u^\prime\phi$ is Lipschitz. Then the map 
\[
\rwt{u^\prime\circ \phi}:\rwt{M^\prime}\to E\pi^{(n-2)}
\]
is a uniformly cobounded Lipschitz map. This implies that $\dim_{MC}M^\prime \le 2n-2$.  Thus, we have constructed a degree one map $f:M^\prime\to M$ that induces an isomorphism of the fundamental groups and $\dim_{MC}\rwt M^\prime\le 2n-2$. By Proposition~\ref{Dranishnikov}, we obtain $\dim_{mc}\Wi M\le\dim_{MC} \Wi M \le 2n-1$.

In the case of complex surfaces, we proceed as follows. Assume $(M^2, g_{\omega})$ is such that $s_{g_{\omega}}>0$. As proved by S.-T.~Yau~\cite[Theorem 2]{Yau74}, the Kodaira--Enriques classification (see, for example,~\cite[Chapter IV]{GH}) gives that $M^2$ is either a \emph{ruled} surface or $\C P^2$. For $\C P^2$, the statement is trivial because we have $\dim_{mc}\C P^2=\dim_{MC}\C P^2=0$. For the remaining cases, recall that a complex surface is ruled if and only if it is the blow-up (possibly at multiple points) of a holomorphic $\C P^1$-bundle over $M_g$. Thus, we have a holomorphic map
\[
\varphi: M^2 \to M_g
\]
with simply connected fibers, which induces an isomorphism on $\pi_1$. By pulling back this map to the universal covers, we conclude that $\dim_{mc}\Wi{M^2}=\dim_{MC}\Wi{M^2}=0$ if and only if $M_g=\C P^1$, and in the remaining cases of genus $g\geq 1$, we obtain $\dim_{mc}\Wi{M^2}=\dim_{MC}\Wi{M^2}=2$. 
\end{proof}
\begin{rem}
In Theorem~\ref{uniruledT}, we would get the full Gromov Conjecture~\ref{G} if Rudyak's Conjecture~\ref{Rudyak} for $\dim_{mc}$ were true for birational maps. Also, in view of the main theorem in~\cite{BD2}, the inequality $\dim_{mc}\Wi{M}\leq2n-2$ holds in Theorem~\ref{uniruledT} for totally non-spin manifolds $M$.
\end{rem}

\begin{rem}
As a by-product of the proof of Theorem~\ref{uniruledT}, we have the equality $\dim_{mc}=\dim_{MC}$ for projective surfaces admitting a K\"ahler metric with positive scalar curvature. Our discovery that certain symmetric squares of surfaces have $\dim_{mc}\neq\dim_{MC}$ (see Theorem~\ref{cex}) implies that this result does not extend to higher dimensions starting with threefolds. More precisely, the projective threefolds
\[
SP^2(M_g)\times \C P^1, \quad g\geq 3,
\]
admit K\"ahler metrics with positive scalar curvature but satisfy $\dim_{mc}\neq \dim_{MC}$.
\end{rem}

\section{Spin structures}\label{spinstr}
The goal of this section is to determine the existence and non-existence of spin structures on symmetric products of surfaces and their universal covers. 

We begin by recalling that an orientable manifold $M$ has a spin structure if and only if its second Stiefel--Whitney class $w_2(M)\in H^2(M; \Z_2)$ vanishes, see~\cite[Theorem 2.1, Page 86]{SpinG}. Clearly, $w_2(M)=0$ if and only if its evaluation on any $2$-dimensional $\Z_2$-homology class of $M$ of is zero. We recall that complex projective space $\mathbb CP^m$ is spin if and only if $m$ is odd.

By the Hurewicz theorem, the universal cover $\wt M$ has a spin structure if and only if $w_2(M)$ vanishes on every spherical $\Z_2$-homology class in $M$. 

\begin{definition}\label{totnonspin}
An orientable manifold $M$ is said to be \emph{totally non-spin} if both $M$ and its universal cover $\Wi{M}$ do not support a spin structure. 
\end{definition}

\begin{ex}\label{ex1}
 In view of the fiber bundle $SP^n(M_g)\to T^{2g}$ with the fiber $\mathbb CP^{n-g}$ for $n\ge 2g-1$, the universal cover of $SP^n(M_g)$ is diffeomorphic to the product $\mathbb CP^{n-g}\times\mathbb R^{2g}$. Therefore, for $n\ge 2g-1$, the manifold $SP^n(M_g)$ is totally non-spin if and only if $n-g$ is even. 
\end{ex}
\begin{ex}\label{ex2}
The manifold $SP^2(M_g)$ is not spin for any $g\ge 0$.  This follows from the fact that every closed spin $4$-manifold has an even intersection form (see~\cite[Theorem 2.10]{SpinG} for a proof). We note that  $(c^*)^2=1$ for $c^*\in H^2(SP^n(M_g))$ in the notations of Section~\ref{cohomo}.
\end{ex}

In this section, we show that the restriction $n\ge 2g-1$  in Example~\ref{ex1} can be dropped and that the conclusion of Example~\ref{ex2} holds for all $n$.

\subsection{The spherical homology class}\label{subsectsph}
We consider the base point inclusion map $\xi:M_g\to SP^n(M_g)$. Let $\phi:\partial D^2\to\bigvee^{2g}S^1$ be the attaching map for the standard CW complex structure on 
\[
M_g=\bigvee^{2g} S^1\cup_\phi D^2.
\]
Since $\pi_1(SP^n(M_g))=\mathbb Z^{2g}$ is abelian, the composition $\xi\circ\phi$ is null-homotopic. Moreover, the null-homotopy can be chosen in $SP^n(\bigvee^{2g}S^1)\subset SP^n(M_g)$. Such a null-homotopy, together with the characteristic map $\psi:D^2\to M_g$, defines 
a map of a $2$-sphere $\alpha:S^2\to SP^n(M_g)$. We denote the corresponding spherical homology class by $u$.

\begin{prop}\label{s-class}
$u=c-\sum a_i\cdot b_i$.
\end{prop}

\begin{proof}
We can view the spheroid $\alpha(S^n)$ as two $2$-cells attached to $\bigvee_{2g}S^1$ with the same attaching map. The first cell is attached by means of the attaching map in $M_g\subset SP^n(M_g)$, which is the product of commutators. Let us identify $SP^n(\bigvee_{2g}S^1)$ with $SP^n(\bigvee_{i=1}^g(S^1_{a_i}\bigvee S^1_{b_i}))$. Then, see that the space $SP^n(\bigvee_{2g}S^1)$ contains the $2$-torus $T_i$ for each $i\le g$ as follows:
\[
T_i^2=S^1_{a_i}\times S^1_{b_i}\subset SP^2\bigl(S^1_{a_i}\bigvee S^1_{b_i}\bigr)\subset SP^n\bigl(S^1_{a_i}\bigvee S^1_{b_i}\bigr)\subset SP^n\bigl(\bigvee_{i=1}^g\bigl(S^1_{a_i}\bigvee S^1_{b_i}\bigr)\bigr).
\]
We defined the second $2$-cell to be attached in $SP^n(\bigvee_{2g}S^1)$ again by the product of the commutators. Thus, the homology class defined by the second cell is the sum of the fundamental classes of tori $[T_i]=a_i\cdot b_i$. This implies the formula.  
\end{proof}

A more formal proof of Proposition~\ref{s-class} can be found in~\cite[Lemma 5]{Ka2}. 
\begin{cor}\label{image of u}
For the map $\bar q: SP^n(M_g)\to \mathbb CP^n$ defined in Section~\ref{cohomo}, we have that $\bar q_*(u)=\bar q_*(c)=[S^2]\in H_2(\mathbb CP^n)$.
\end{cor}

\begin{proof}
 Since the support of the cycle $a_i\cdot b_i$ lies in $SP^n(\bigvee_{2g}S^1)$, we have $\bar q_*(a_i\cdot b_i)=0$ for each $i$. Thus, we use Proposition~\ref{s-class} to get $\bar q_*(u)=\bar q_*(c)=[S^2]$.
\end{proof}

\begin{prop}\label{generator}
    For $n\ge 3$, $\pi_2(SP^n(M_g))$ is generated by $\alpha:S^2\to SP^n(M_g)$.
\end{prop}

\begin{proof}
By Proposition~\ref{pi2}, $\pi_2(SP^n(M_g))=\mathbb Z$. Since the homology class $u$ belongs to the image of the Hurewicz homomorphism $h:\pi_2(SP^n(M_g))\to H_2(SP^n(M_g))$, it suffices to show that $u$ is not divisible in $H_2(SP^n(M_g))$.
Assume that $u=kv$ for some $v\in H_2(SP^n(M_g))$. Since $u=c-\sum_ia_i\cdot b_i$, we get by Proposition~\ref{2-homology} that 
\[
c-\sum_{i=1}^g a_i\cdot b_i=k\left(r c+\sum_{i<j}m_{i,j}\hspace{1mm}a_i\cdot a_j+\sum_{i<j}n_{i,j}\hspace{1mm}b_i\cdot b_j+\sum_{i, j}\ell_{i,j}\hspace{1mm}a_i\cdot b_j\right).
\]
Then, we obtain the equality
\[
(kr-1)c+k\left(\sum_{i<j}m_{i,j}\hspace{1mm}a_i\cdot a_j+\sum_{i<j}n_{i,j}\hspace{1mm}b_i\cdot b_j+\sum_{i\ne j}\ell_{i,j}\hspace{1mm}b_i\cdot a_j\right)+\sum_{i=1}^g (k\ell_{i,i}+1)\hspace{1mm}a_i\cdot b_i=0,
\]
which implies that $kr=1$. Hence, we get $k=\pm 1$.
\end{proof}

\begin{prop}\label{0 on spherical}
    Let $n\ge 2$ and $g>0$. Then the evaluation of the class $\sum_ia_i^*b_i^*$ on any spherical homology class $v\in H_2(SP^n(M_g))$ is zero. 
\end{prop}

\begin{proof}
    From Proposition~\ref{2-homology}, we have for each $1\le i \le g$ that $a_i^*=\mu_n^*(\alpha_i)$ and $b_i^*=\mu_n^*(\beta_i)$ for some $\alpha_i,\beta_i\in H^1(J)$, where 
    \[
    \mu_n:SP^n(M_g)\to J
    \]
    is the Abel--Jacobi map (in Section~\ref{cohomo}, we used the same notation for $a_i^*$ and $\alpha_i$, and for $b_i^*$ and $\beta_i$). Let $\gamma:S^2\to SP^n(M_g)$ be the map corresponding to $v$. Since any map from a $2$-sphere to a torus is null-homotopic, the composition $\mu_n\circ\gamma:S^2\to J$ is null-homotopic. Hence, $(\mu_n)_*(v)=0$. Thus, in the zeroth homology, we have that
\[
(\mu_n)_*\left(v\frown \sum_{i=1}^ga_i^*b_i^*\right)=(\mu_n)_*\left(v\frown \mu_n^*\left(\sum_{i=1}^g\alpha_i\beta_i\right)\right)=(\mu_n)_*(v)\frown \sum_{i=1}^g\alpha_i\beta_i=0.
\]
The conclusion now follows upon observing that $\mu_n$ induces an isomorphism of the zeroth homology groups due to Corollary~\ref{mu}. 
\end{proof}

\subsection{Determining the spin}
Recall that for any fixed $n$ and $g$, the second Stiefel--Whitney class of $SP^n(M_g)$, denoted $w_2$, is the $\Z_2$-cohomology class
\[
w_2=\left((n-g+1)c^*-\sum_{i=1}^ga_i^*b_i^*\right) \text{ mod } 2.
\]

\begin{prop}\label{odd-even}
    Let $n-g$ be odd. Then the universal cover of $SP^n(M_g)$ is spin.
\end{prop}

\begin{proof}
    First, let $g>0$. Since $n-g$ is odd, the Stiefel--Whitney class of $SP^n(M_g)$ equals $w_2=\sum_ia_i^*b_i^*$. By Proposition~\ref{0 on spherical}, the evaluation of $w_2$ on any spherical class $v\in H_2(SP^n(M_g))$ is trivial for each $g>0$. Finally, we mention that for $g=0$ and $n$ odd, $\C P^n=SP^n(M_0)$ is spin. Hence, the conclusion follows.
\end{proof}

\begin{prop}\label{odd2}
    Let $n-g$ be odd. Then $SP^n(M_g)$ is spin if and only if $g=0$. 
\end{prop}

\begin{proof}
    When $g=0$, $SP^n(M_0)=\C P^n$ is spin for odd $n$. So, let us assume $g>0$. We show that the evaluation of $w_2=\sum_ia_i^*b_i^*$ on the Pontryagin product $a_1\cdot b_1$ is non-trivial. 
    
    We recall that $\bar q_*(a_i\cdot b_i)=0$ because the support of the cycle $a_i\cdot b_i$ lies in $SP^n(\bigvee_{2g}S^1)$. See that for each $i$,
    \[
    \bar{q}_*\left((a_i\cdot b_i)\frown c^*\right)=\bar{q}_*\left(a_i\cdot b_i\right)\frown c^*=0.
    \]
    Since $\bar{q}$ induces an isomorphism on zeroth homology, we have $(a_i\cdot b_i)\frown c^*=0$ for all $i$. Then using the formula $(a_i\cdot b_i)^*=a_i^*b_i^*-c^*$ from~\cite{KS}, we obtain the following equality mod $2$:
\[
(a_1\cdot b_1)\frown\sum_{i=1}^ga_i^*b_i^*=\left((a_1\cdot b_1)\frown\sum_{i=1}^g(a_i\cdot b_i)^*\right)-g\left((a_1\cdot b_1)\frown c^*\right)=1.
\]
\end{proof}

\begin{prop}\label{even2}
    Let $n-g$ be even. Then the manifold $SP^n(M_g)$ is totally non-spin for all $g$.
\end{prop}

\begin{proof}
For $g=0$, it is well-known that $SP^n(M_0)=\C P^n$ is not spin if $n$ is even. So, we let $g>0$. For $n-g$ even, the second Stiefel--Whitney class of $SP^n(M_g)$ equals $w_2=c^*+\sum_ia_i^*b_i^*$. We show that the evaluation of $w_2$ on $u$ is non-trivial when $n-g$ is even and $g>0$. In view of Corollary~\ref{image of u}, 
\[
\bar q_*(u\frown c^*)=\bar q_*(u)\frown c^*=[S^2]\frown c^*=1.
\]
Since $\bar q_*$ is an isomorphism of $0$-dimensional homology, we have $u\frown c^*\ne 0$. Then using Proposition~\ref{0 on spherical}, we obtain
\[
u\frown w_2=\left(u\frown c^*\right)+\left(u\frown \sum_{i=1}^ga_i^*b_i^*\right)=u\frown c^*\ne 0.
\]
Therefore, $SP^n(M_g)$ and its universal cover are both not spin for each $g$.
\end{proof}

In this section, we completely determined the existence of spin structures on $SP^n(M_g)$ and its universal cover for each $n\ge 2$ and $g\ge 0$.

\begin{thm}\label{main2}
\begin{enumerate}
    \item The manifolds $SP^n(M_g)$ are never spin for $g>0$. 
    \item The universal covering of $SP^n(M_g)$ is spin if and only if $n-g$ is odd.
\end{enumerate}
\end{thm}


The backward direction of part (2) of Theorem~\ref{main2} can be strengthened to the following.

\begin{thm}\label{spin cover}
Let $n-g$ be odd. Then the manifold $SP^n(M_g)$ admits a $2^g$-sheeted spin covering $p:N'\to SP^n(M_g)$. Moreover, $p$ is the projection onto the orbit space of a free $\mathbb Z_{2^g}$-action.
\end{thm}

\begin{proof}
Let $p':T^{2g}\to T^{2g}$ be a covering map that corresponds to the subgroup $$\mathbb Z\langle2a_1,\dots,2a_g,b_1,\dots,b_g\rangle\subset \mathbb Z\langle a_1,\dots,a_g,b_1,\dots,b_g\rangle=H_1(T^{2g})=\pi_1(T^{2g}).$$ We define $p:N'\to SP^n(M_g)$ to be the pullback of $p'$ with respect to the Abel--Jacobi map 
\[
\mu_n:SP^n(M_g)\to T^{2g}.
\]
We claim that $N'$ is spin. Note that $w_2(N')=p^*(w_2(SP^n(M_g)))$.
Recall that when $n-g$ is odd, we have the mod 2 equality 
$w_2(SP^n(M_g))=\sum_ia_i^*b^*_i$. Thus, in mod $2$, we have
\[
w_2(N')=p^*(w_2(SP^n(M_g)))=(\mu_n')^*(p')^*\left(\sum_{i=1}^ga_i^*b^*_i\right)=(\mu_n')^*\left(\sum_{i=1}^g2a_i^*b^*_i\right)=0,
\]
where $\mu_n':N'\to T^{2g}$ is the map parallel to $\mu_n$ in the pullback diagram. 
\end{proof}

\section{Macroscopic dimension and positive scalar curvature}\label{dimandpsc}

The aim of this section is to study the interplay between Gromov's macroscopic dimension and the existence of metrics of positive scalar curvature on symmetric products of surfaces.

\subsection{Macroscopic dimensions of symmetric products of surfaces}\label{macdimsec}

Below, we estimate the macroscopic dimensions of the symmetric products of surfaces. 
\begin{thm}\label{cex}
For $n\ge g$, $\dim_{mc}\rwt{SP^n(M_g)}=\dim_{MC}\rwt{SP^n(M_g)}=2g$. For $n<g$, 
\begin{enumerate}
\item $\dim_{MC}\rwt{SP^n(M_g)}=2n$; 
\item $\dim_{mc}\rwt{SP^n(M_g)}\le 2n-1$;
\item  $\dim_{mc}\rwt{SP^n(M_g)}\le 2n-2$ when $g-n$ is even.
\end{enumerate}
\end{thm}

\begin{proof}
By Corollary~\ref{newproof}, we note that $\dim_{mc}\rwt{SP^g(M_g)}=2g$. For $n\ge g$, the universal cover 
\[
\Wi\mu_n:\rwt{SP^n(M_g)}\to\Wi{J}=\mathbb R^{2g}
\]
of the Abel--Jacobi map $\mu_n:SP^n(M_g)\to J$ is uniformly cobouded and Lipschitz. Therefore, $\dim_{MC}\rwt{SP^n(M_g)}\le 2g$. Since  $\rwt{SP^g(M_g)}$ is a subspace of $\rwt{SP^n(M_g)}$, we obtain that $2g=\dim_{mc}\rwt{SP^g(M_g)}\le\dim_{mc}\rwt{SP^n(M_g)}$ for all $n\ge g$. This implies the first string of equalities.

Note that $\pi_1(SP^n(M_g))=\Z^{2g}$ is amenable. By Theorem~\ref{essential}, $SP^n(M_g)$ is rationally essential for $n\le g$. Thus, the second equality follows from~\cite[Theorem 7.2]{Dr1}, which states that $\dim_{MC}\Wi{M} = k$ for a closed rationally essential $k$-manifold $M$ with $\pi_1(M)$ amenable.

Theorem 6.3 in~\cite{Dr3} states that $\dim_{mc} \Wi{M} < k$ for a closed rationally essential $k$-manifold $M$ if $\pi_1(M)$ is a geometrically finite amenable duality group with $\cd(\pi_1(M))>k$. Since all these conditions are satisfied for $SP^n(M_g)$ in the case $n<g$, we obtain $\dim_{mc}\rwt{SP^n(M_g)}\le 2n-1$ in this case.

The last inequality follows from the third, Theorem~\ref{main2}, and the main result of~\cite{BD2}, which states that for a closed totally non-spin $m$-manifold for $m\ge 5$, the inequality $\dim_{mc}\Wi{X}\le m-1$ implies the inequality $\dim_{mc}\Wi{X}\le m-2$. This covers the case $n\ge 3$. Since $\pi_1(SP^2(M_g))$ is residually finite, the case $n=2$ in the last inequality is covered by~\cite[Theorem 1.2]{DD}.
\end{proof}

In~\cite{Gr2}, Gromov conjectured the following.

\begin{conjec}[\protect{Rational Inessentiality Conjecture}]\label{grconj2}
If $M$ is a closed orientable $n$-manifold with $\dim_{mc}\wt{M}<n$, then $M$ is not rationally essential. 
\end{conjec}

First counterexamples to Conjecture~\ref{grconj2} in dimensions $n\ge 4$ were constructed in~\cite{Dr2}. Theorem~\ref{cex} produces more counterexamples to Conjecture~\ref{grconj2}. Namely, the $2n$-dimensional manifolds $SP^n(M_g)$ are counterexamples to Gromov's conjecture for all $g > n$. To get examples in odd dimensions $\ge 5$, we take products of these manifolds with $S^1$. Notably, unlike the counterexamples of ~\cite{Dr2}, the ones we produce here have amenable fundamental groups.

We note that Gromov's conjecture holds in lower dimensions.

\begin{prop}\label{lowdim2}
In dimension $3$, Conjecture~\ref{grconj2} is true.
\end{prop}

\begin{proof}
In the proof of Theorem~\ref{lowdim}, we show that $\dim_{mc}\wt{M}=3$ if $M$ is a closed aspherical $3$-manifold, or if $M$ has an aspherical component in its prime decomposition. For all other closed $3$-manifolds $M$, we verify that $\dim_{mc}\wt{M}\leq 1$.  As shown in~\cite[Theorem 3]{KN13}, a closed $3$-manifold $M$ is rationally essential if and only if it is either aspherical or it has an aspherical component in its prime decomposition. By combining these results, we have a proof of Conjecture~\ref{grconj2} in dimension $3$.
\end{proof}

\begin{remark}\label{mc and MC}
    In~\cite[Section 6]{Dr3}, a rationally essential closed $k$-manifold $M$ having a geometrically finite amenable fundamental group $\pi_1(M)$ and satisfying 
    \[
    \dim_{mc}\wt{M}<\dim_{MC}\wt{M}=k
    \]
    was produced in each dimension $k\ge 5$ using surgery (see also~\cite[Corollary 7.3]{Dr1}). Here, in each even dimension $\ge 4$, Theorem~\ref{cex} brings an infinite family such examples, namely $SP^n(M_g)$  for all $g>n\ge 2$. The products of these manifolds with $S^1$ give such examples in all odd dimensions $\ge 5$.
\end{remark}

We note that the examples in Remark~\ref{mc and MC} are sharp in view of Theorem~\ref{lowdim}.

\subsection{Positive scalar curvature}\label{psc section}
Let us now use the machinery developed so far to answer some questions on the existence and the non-existence of metrics of positive scalar curvature on the symmetric products of surfaces. Throughout this section, we use the abbreviation PSC for positive scalar curvature.

We begin by stating some definitions and results that will be used in the proof of our main result of this section.

A compact $k$-manifold $X$ is called \emph{enlargeable}, in the sense of Gromov and Lawson~\cite{GL0}, if for each $\varepsilon>0$, there exists a finite cover $X_{\varepsilon}$ of $X$ and a map $f\colon X_{\varepsilon}\to S^k$ of non-zero degree that is $\varepsilon$-contracting. Examples of such manifolds include solvmanifolds and hyperbolic manifolds, as well as their finite products and connected sums.


\begin{thm}[Gromov--Lawson,~\cite{GL}]\label{hyper}
An enlargeable spin manifold cannot support a PSC metric.
\end{thm}

Let $H_k(B\pi)^+$ denote the set of images of fundamental classes $f_*([N^k])\in H_k(B\pi)$  for maps $f:N^k\to B\pi$ of orientable PSC manifolds $N^k$ of dimension $k$.
We note that $0\in H_*(B\pi)^+$ for all $\pi$.

The following theorem was proven in the recent paper~\cite[Theorem A.1]{GrH}; the sketch of its proof for manifolds with no 2-torsion in homology (our case) was given in~\cite[Theorem 4.11]{RS}.

\begin{thm}\label{sjth}
Let $M$ be an orientable manifold of $\dim\geq 5$ with fundamental group $\pi=\pi_1(M)$ and classifying map $u:M\to B\pi$. Suppose that the universal cover of $M$ is not spin and $u_*([M])\in H_*(B\pi)^+$. Then $M$ admits a PSC metric.
\end{thm}

The following theorem is a combination of results from~\cite{SY1} developed further in~\cite{Sch} and~\cite{JS}. 
\begin{thm}\label{SY}
Suppose that a closed orientable $n$-manifold $M$, $n\le 8$, has cohomology classes $\alpha_1,\dots,\alpha_{n-1}\in H^1(M)$ with non-zero cup product $\alpha_1\cdots\alpha_{n-1}\ne 0$. Then $M$ cannot support a PSC metric.
\end{thm}

Before stating our result on PSC metrics, we also recall the following.

\begin{prop}[\protect{\cite[Proposition 4.6]{BD2}}]\label{map to torus}
Suppose that a map $u : M\to T^m $ of a closed
oriented $n$-manifold $M$ to an $m$-torus takes the fundamental class $[M]$ to a non-zero element in $H_n(T^m)$. Then there exists a map $f:M\to T^n$ whose degree is non-zero.
\end{prop}

We can now discuss the existence and non-existence of PSC metrics on the symmetric products of surfaces.

\begin{thm}\label{main3}
\rm{(Non-existence of PSC)} {\em The manifolds $SP^n(M_g)$ do not admit a PSC metric for $n\le g$ when} 

\begin{enumerate}
    \item $n-g$ {\em is odd};
    \item  $n\le \min\{g,4\}$.
\end{enumerate}

\rm{(Existence of PSC)} {\em The manifolds $SP^n(M_g)$  admit a PSC metric when}
\begin{enumerate}
    \item $n\ge 2g-1$;
    \item $n>g$ \it{and} $n-g$ \it{is even}.
\end{enumerate}
\end{thm}

\begin{proof} \emph{Proof of Non-existence.}
 
(1) We note that the only the case $n\ge 2$ is non-trivial. In this range, under our hypotheses, $SP^n(M_g)$ is rationally essential (\emph{cf}. Theorem~\ref{essential}), and it admits a finite cover $N'$ which is spin (\emph{cf}. Theorem~\ref{spin cover}). It is easy to see that $N'$ is also rationally essential and we have isomorphisms $\pi_1(N')=\pi_1(SP^n(M_g))=\Z^{2g}$. It then follows from~\cite[Theorem~4.1~(i)]{BH} that $N'$ is enlargeable. Therefore, by the Gromov--Lawson Theorem~\ref{hyper}, $N'$ (and hence $SP^n(M_g)$) cannot support a metric with positive scalar curvature.

(2) In view of (1), it only remains to cover the case when $n-g$ is even. We note that in this case, the manifold $SP^n(M_g)$ is totally non-spin by Theorem~\ref{main2}. Since $n\le g$, $SP^n(M_g)$ is essential by Theorem~\ref{essential}. Then by Proposition~\ref{map to torus}, there is a map 
\[
f:SP^n(M_g)\to T^{2n}
\]
of non-zero degree. We now apply Theorem~\ref{SY} for $2n\le 8$ to get the result.

\emph{Proof of Existence.}

(1) As shown in Theorem~\ref{n ge 2g}, $SP^n(M_g)$ admits a K\"ahler metric of positive scalar curvature for $n\ge 2g-1$.

(2) By Theorem~\ref{main2}, the universal cover of $SP^n(M_g)$ is not spin for $n-g$ even. Since $n>g$, we have
\[
\mu_*\left(\left[SP^n(M_g)\right]\right)\in H_{2n}\left(T^{2g}\right)=0.
\]
Hence, the conclusion follows from Theorem~\ref{sjth}.
\end{proof}

We note that when $n=g$, then regardless of the dimension, $SP^n(M_n)$ cannot support a K\"ahler PSC metric due to Corollary~\ref{nequalg}.

\begin{rem}
    As explained in the proof of Theorem~\ref{main3}~(1), since $SP^n(M_g)$  is rationally essential with free abelian fundamental group in the range $2\le n\le g$, it is enlargeable. However, for $n<g$, because $\cd(\pi_1(SP^n(M_g)))>\dim(SP^n(M_g))$, it follows from~\cite[Theorem~4.1~(ii)]{BH} that the universal cover of $SP^n(M_g)$ is neither \emph{hypereuclidean}, nor \emph{hyperspherical}, in the sense of~\cite{Gro93}. Of course, the same conclusions hold for $SP^n(M_g)\times S^1$ when $n<g$. Thus, we have natural examples of enlargeable manifolds (in each dimension $\ge 4$) that are macroscopically small with respect to several classical notions of ``largeness'' (see also Theorem~\ref{cex}~(2)).
\end{rem}

In the range $g<n<2g-1$ when $n-g$ is odd, the problem of the existence of a PSC metric on $SP^n(M_g)$ is less clear.
Perhaps it can be resolved for finite covers of $SP^n(M_g)$. We recall that
Theorem~\ref{spin cover} implies that when $n-g$ is odd, the manifold $SP^n(M_g)$ admits a $KO$-orientable $2^g$-folded cover 
\[
p:N'\to SP^n(M_g).
\]
It has a canonical map $\mu_n':N'\to T^{2g}$ which is the pullback of the Abel--Jacobi map $\mu_n$ with respect to a $2^g$-fold cover of the torus $T^{2g}$.

\begin{question}
For $n-g$ odd, is the image of the $KO$-fundamental class $[N']_{KO}$ zero in $KO_*(T^{2g})$?
\end{question}

If this question can be answered in the affirmative, then by the Rosenberg--Stolz theorem~\cite[Theorem 4.13]{RS}, the manifold $N'$ will admit a metric of positive scalar curvature. We note that in view of Theorem~\ref{main3} (1), this can potentially happen only for $n>g$.

\section{Converse to Gromov's conjecture}\label{converse}

In this section, we speculate on a possible converse to Gromov's Conjecture~\ref{G}. In dimensions two and three, it follows from Theorem~\ref{lowdim} and the classification of $2$- and $3$-manifolds that
\[
\dim_{mc}\Wi{M}=\dim_{MC}\Wi{M}<\dim M \Longleftrightarrow M\text{ admits a PSC metric!}
\]
Emboldened by this observation, one may ask the following question. 
\begin{question}\label{grconverse}
    Let $M$ be a closed orientable $n$-manifold. If $\dim_{mc}\Wi M \le n-2$, is it true that $M$ admits a metric of positive scalar curvature?
\end{question}

Of course, a bit of thought reveals that this question cannot possibly have a positive answer in higher dimensions. This is mainly due to the fact that simply connected closed manifolds exhibit a much richer structure. Indeed, this goes all of the way back to the beginning of the field of Spin Geometry with the proof by Lichnerowicz that the so-called $K3$ surface does not support metrics of positive scalar curvature,~\cite[Chapter II, Section 8]{SpinG}. Similarly, all non-singular hypersurfaces in $\C P^3$ of even degree $\ge 4$ are examples of spin, simply connected $4$-manifolds supporting no metrics with positive scalar curvature. In dimension four, Seiberg--Witten theory 
removes the spin assumption from many of these theorems. Indeed, any surface of general type cannot support a metric of positive scalar curvature, and moreover, it has a negative Yamabe invariant that can be explicitly computed. We refer to the paper of LeBrun~\cite{LeBrun} for the proof of this striking result. 

Similarly, for dimensions $n\geq 5$, in view of the index obstruction~\cite{RS} for spin manifolds, there are many simply connected examples that do not admit PSC metrics. All such manifolds $M$ have $\dim_{mc}\Wi M=0$, so the converse of Conjecture~\ref{G} does not seem to be a route to be pursued. On the other hand, this may seem to be too pessimistic. Indeed, it is known that non-spin simply connected $n$-manifolds admit PSC metrics for $n\ge 5$,~\cite{GL1}. Thus, the converse of Gromov's conjecture could make sense in the realm of totally non-spin manifolds. Unfortunately, the following set of examples arise from Section~\ref{dimandpsc}.

\begin{ex}\label{convnottrue}
Let $n\in \{2,3,4\}$. If $g>n$ such that $g-n$ is even, then the closed orientable totally non-spin $2n$-manifold $SP^n(M_g)$ cannot support a metric of positive scalar curvature because of Theorem~\ref{main3}, even though 
    \[
    \dim_{mc}\rwt{SP^n(M_g)}\le 2n-2
    \]
    due to Theorem~\ref{cex}. It is easy to see that in these cases, one has
    \[
    \dim_{mc}\rwt{SP^n(M_g)\times S^1}\le 2n-1.
    \]
If $n\in \{2,3\}$ and $g>n$ such that $g-n$ is even, then using Corollary~\ref{essalldim} we can proceed exactly as in the proof of Theorem~\ref{main3} to deduce that the manifold $SP^n(M_g)\times S^1$ cannot support a PSC metric. Concretely, closed manifolds that answer Question~\ref{grconverse} in negative are
\begin{itemize}
    \item in dimension $4$: $SP^2(M_g)$ for $g\ge 4$ even,
    \item in dimension $5$: $SP^2(M_g)\times S^1$ for $g\ge 4$ even,
    \item in dimension $6$: $SP^3(M_g)$ for $g\ge 5$ odd,
    \item in dimension $7$: $SP^3(M_g)\times S^1$ for $g\ge 5$ odd, and
    \item in dimension $8$: $SP^4(M_g)$ for $g\ge 6$ even.
\end{itemize}
\end{ex}

Note that Bolotov~\cite[Corollary 2.2]{Bo09} constructs interesting examples of spin $n$-manifolds $M$ with $\pi_1(M)$ non-amenable and with no PSC metrics that satisfy $\dim_{mc}\Wi{M}=n-1$.

The examples we discussed suggest that it may be more reasonable to consider the converse to Gromov's conjecture for the macroscopic dimension $\dim_{MC}$. We can give some supporting evidence for this general strategy in the following result. 

\begin{thm}\label{totallyconverse}
Let $M$ be a closed totally non-spin $n$-manifold with amenable fundamental group $\pi$. If $H_n(\pi)$ is torsion-free, then the inequality $\dim_{MC}\Wi M\le n-1$ implies the existence of a PSC metric. 
\end{thm}

\begin{proof}
By~\cite[Theorem 7.2]{Dr1}, we know that every rationally essential $n$-manifold with amenable fundamental group satisfies $\dim_{MC}\Wi M=n$. Since the group $H_n(\pi)$ is torsion-free, we obtain that $M$ is inessential. By Theorem~\ref{sjth}, $M$ admits a metric of positive scalar curvature.
\end{proof}
We conclude by pointing out that in view of~\cite[Theorem 5.2]{Dr2}, it seems hard to generalize this theorem beyond the realm of amenable fundamental groups.

\section*{Acknowledgment}
LFDC thanks Rita Pardini for introducing him to symmetric squares of surfaces, Roberto Svaldi for expert advice on the latest developments in the minimal model program, and Mikhail Gromov for a useful email exchange several years ago.  He also thanks Claude LeBrun and Alexandru Suciu for useful discussions during the 2024 Joint Meeting of the New Zealand, Australian, and American Mathematical Societies, and Saman Esfahani and Adam Levine for pointing out the relevance of symmetric products of surfaces in Heegaard Floer homology. He was supported in part by NSF grant DMS-2104662. 

AD thanks Bernhard Hanke for useful bibliographical suggestions. He also thanks FIM, the Institute of Mathematical Research at ETH, Zurich, and Max-Planck Institut f\"ur Mathematik, Bonn, for hospitality. He was supported in part by Simons Foundation AWD-625962.

EJ thanks Bernhard Hanke for useful discussions on enlargeable manifolds, and Aditya Kumar for conversations about Question~\ref{grconverse} in dimension $4$. 

\end{document}